\documentclass[reqno]{article}
\usepackage{amsmath,amsfonts,amssymb,amsthm}
\usepackage{graphics,color,epsfig, enumerate}
\usepackage{tabularx, pbox}

\usepackage{tikz}
\tikzstyle{connector} = [->,thick]
\usetikzlibrary{calc}
\pgfdeclarelayer{edgelayer}
\pgfdeclarelayer{nodelayer}
\pgfsetlayers{edgelayer,nodelayer,main}

\newtheorem{theorem}{Theorem}
\newtheorem{lemma}[theorem]{Lemma}

\newtheorem{corollary}[theorem]{Corollary}

\newtheorem{question}{Question}

\theoremstyle{remark}

\begin{document}
\title{Graphs of large linear size are antimagic}

\author{Tom Eccles
\thanks{eccles.tom@gmail.com}
}
\date{\today}
\maketitle

\begin{abstract}
Given a graph $G=(V,E)$ and a colouring $f:E\mapsto \mathbb N$, the induced colour of a vertex $v$ is the sum of the colours at the edges incident with $v$. If all the induced colours of vertices of $G$ are distinct, the colouring is called antimagic. If $G$ has a bijective antimagic colouring $f:E\mapsto \{1,\dots,|E|\}$, the graph $G$ is called antimagic. A conjecture of Hartsfield and Ringel states that all connected graphs other than $K_2$ are antimagic. Alon, Kaplan, Lev, Roddity and Yuster proved this conjecture for graphs with minimum degree at least $c \log |V|$ for some constant $c$; we improve on this result, proving the conjecture for graphs with average degree at least some constant $d_0$.
\end{abstract}

\section{Introduction}\label{sec_intro}
All graphs in this paper are simple and undirected, except where we explicitly state otherwise. By a \emph{colouring} of a set $S$, we mean a function $f:S\mapsto \mathbb N$. For $s\in S$, $f(s)$ is called the \emph{colour} of $s$. We call $f$ a \emph{labelling} if it is injective, and in this case $f(s)$ is called the \emph{label} of $s$. For a graph $G$ and a colouring $f:E(G)\mapsto \mathbb N$, the \emph{induced colour} of a vertex $v$ is the sum of the colours of the edges incident with $v$. The colouring $f$ is called \emph{antimagic} if the induced colours at different vertices are distinct. If a graph $G$ admits a bijective antimagic labelling $f:E(G)\mapsto \{1,\dots,|E(G)|\}$, then we call $G$ \emph{antimagic}.

Hartsfield and Ringel \cite{HaRi} conjectured that all connected graphs on at least $3$ vertices are antimagic. This problem remains open, but there are numerous partial results. Hefetz \cite{Hef} proved that a graph on $3^k$ vertices which admits a $C_3$-factor is antimagic. This was generalised by Hefetz, Saluz and Tran \cite{HeSaTr}, who proved that a graph on $p^k$ vertices admitting a $C_p$-factor is antimagic. Cranston \cite{Cra} proved that any regular bipartite graph is antimagic. Perhaps the most significant result on antimagic graphs is that of Alon, Kaplan, Lev, Roddity and Yuster \cite{AKLRY}, who proved that there is an absolute constant $c_0$ such that if $G$ is a graph on $n$ vertices with minimum degree at least $c_0 \log n$ then $G$ is antimagic. For more information on this and related labelling problems, see the survey paper \cite{Gal}.

Our main theorem is an improvement on the result of \cite{AKLRY}. Note that if a graph $G$ has two isolated vertices, or any isolated edge, it cannot be antimagic. However, we shall show that if a graph $G$ has large average degree while avoiding these trivial obstacles, $G$ is antimagic.
\begin{theorem}\label{thm_main}
 There exists an absolute constant $d_0$ so that if $G$ is a graph with average degree at least $d_0$, and $G$ contains no isolated edge and at most one isolated vertex, $G$ is antimagic.
\end{theorem}
The rest of the paper will be organised as follows. In Section \ref{sec_lemmas}, we prove some preliminary lemmas which will be needed during the proof of Theorem \ref{thm_main}. Sections \ref{sec_av_to_min}, \ref{sec_min_to_form} and \ref{sec_form_to_anti} give the proof of Theorem \ref{thm_main}. In Section \ref{sec_av_to_min}, we shall reduce the problem of finding an antimagic labelling for a graph with large average degree to a similar problem for a graph with minimum degree at least some constant. In Section \ref{sec_min_to_form}, we shall put a graph with large minimum degree in a special form, and in Section \ref{sec_form_to_anti} we shall label a graph in this form. In Section \ref{sec_further} we shall discuss possible directions for further work.

\section{Preliminary Lemmas}\label{sec_lemmas}
In this section we shall prove or recall various results which will be needed in the proof of Theorem \ref{thm_main}. The reader who is not overly concerned with the technical details of the proof may wish only to skim this section, referring back to it as necessary during the proof.

In Subsection \ref{subsec_graph_lemmas} we shall prove some simple results about graphs. In Subsection \ref{subsec_dominating} we recall the definition of a total dominating set, and quote a theorem about the size of the total $k$-domination number of a graph with large minimum degree. These two subsections contain lemmas which will be used in Section \ref{sec_min_to_form}, in which we take a graph with large minimum degree and partition the edges and vertices in a certain way. In Subsection \ref{subsec_k_colourings} we prove four technical lemmas about edge colourings of a graph modulo $k$ for some integer $k$; these lemmas will be needed in Section \ref{sec_form_to_anti}, when we shall label the edges of a graph in the form guaranteed by Section \ref{sec_min_to_form}.

\subsection{Graph Lemmas}\label{subsec_graph_lemmas}
In this subsection, we prove two basic results on graphs. Lemma \ref{lem_all_colours} is a result about colouring a graph so that every colour appears at every vertex, and Lemma \ref{lem_random_partition} concerns finding a bipartition of a graph with many edges, so that each part has many edges. Corallary \ref{Cor_good_partition} is simply a special case of Lemma \ref{lem_random_partition} --- this is the form we shall find useful later.

We start with a well known lemma about equitable bipartitions of graphs. An edge-colouring of a graph $G$ is called equitable if for every vertex $v$, the numbers of edges indicident at $v$ which receive each colour differ by at most $1$.

\begin{lemma}\label{lem_all_colours}
Let $G=(V,E)$ be a graph with minimum degree at least $2k+1$. Then $G$ has an edge-colouring $f:E\mapsto \{1,2\}$ such that every vertex is contained in at least $k$ edges of each colour.
\end{lemma}
\begin{proof}
We may assume $G$ is connected; if not, we just consider each component separately. We pair up the vertices of $G$ of odd degree, and join each pair with an extra edge to form a multigraph $G'$. Since all the degrees of vertices in $G'$ even, $G'$ has an Eulerian circuit $C$ -- that is, a walk which begins and ends at the same vertex, and contains each edge exactly once. If any extra edge was added to $G$ to form $G'$, we choose $C$ so to start with with such an edge. Now, we colour the edges of $C$ alternately $1$ and $2$; each vertex is then contained in an equal number of edges of each colour, except the starting vertex of the walk, which may have $2$ more edges coloured $1$ than $2$. When restricted to $G$, this colouring is equitable unless every degree is even, in which case there may be exactly one vertex with exactly $2$ more edges of one colour than the other. Since $G$ has all degrees at least $2k+1$, in this colouring every vertex has at least $k$ incident edges of each colour.
\end{proof}

Next we shall prove a result about partitioning the vertices of a graph with many edges into two vertex classes, each having many edges --- this will be used in Section \ref{sec_min_to_form}. We define $m(n,r_1,r_2)$ to be the least $r$ such that every graph $G$ on $n$ vertices with $r$ edges has a vertex partition $V(G)=V_1\cup V_2$ with at least $r_1$ edges contained in $V_1$, and at least $r_2$ edges contained in $V_2$. If even $r = \binom{n}{2}$ does not suffice, for convenience we set $m(n,r_1,r_2)$ to be $\binom{n}{2}+1$. We bound $m(n,r_1,r_2)$ simply by considering the number of edges in each half of a random partition of $V$.

\begin{lemma}\label{lem_random_partition}
Let $n$, $r_1$ and $r_2$ be positive integers, and for $i=1$, $2$ let $p_i=\frac{\sqrt{r_i}}{\sqrt{r_1}+\sqrt{r_2}}$. Suppose that $r$ is an integer such that
\[
 rp_i^2 - \sqrt{rp_i^2+ 2rnp_i^3-r(2n+1)p_i^4}\ge r_i
\]
holds for $i=1,\,2$. Then $m(n,r_1,r_2)\le r$.
\end{lemma}
\begin{proof}
Let $G=(V,E)$ be a graph with $n$ vertices and $r$ edges --- our task is to find a partition of $V$ with at least $r_1$ edges in one part, and at least $r_2$ in the other. We take a random partition $V_1$, $V_2$ of $V$, with vertices placed independently with probability $p_i$ of being in $V_i$. Let $X_1$, $X_2$ be the random variables corresponding to the numbers of edges contained in $V_1$, $V_2$ respectively. For $i=1$, $2$ let $\mu_i$, $\sigma_i$ and $m_i$ be the mean, standard deviation and median of $X_i$ respectively. It is enough to show that for $i=1$, $2$ we have $m_i > r_i$; then with positive probability we have $X_i>r_i$ for $i=1$, $2$. Now, the mean of $X_i$ is $\mu_i = rp_i^2$, and the variance is given by
\begin{align*}
\sigma^2_i &= \sum_{e_1,\,e_2\in E(G)}\mathbb{P}(e_1\text{ and }e_2\in E_G(V_i))-\mathbb{P}(e_1\in E_G(V_i))\mathbb{P}(e_2\in E_G(V_i))\\
&= r(p_i^2-p_i^4)+\sum_{v\in V}d_G(v)(d_G(v)-1)(p_i^3-p_i^4)\\
&< r(p_i^2-p_i^4)+\frac{2r}{n}n(n-1)(p_i^3-p_i^4)\\
&< r(p_i^2-p_i^4)+2rn(p_i^3-p_i^4)\\
&= rp_i^2+ 2rnp_i^3-r(2n+1)p_i^4.
\end{align*}
Now, the mean and the median of a random variable differ by at most the standard deviation, and so
\begin{align*}
m_i &\ge \mu_i-\sigma_i\\
&> rp_i^2 - \sqrt{rp_i^2+ 2rnp_i^3-r(2n+1)p_i^4}\\
&\ge r_i.
\end{align*}
This proves the claim.
\end{proof}
We shall apply this in a specific case. If $r=an$, and $r_i=a_in$ for $i=1$, $2$, then $p_i = \frac{\sqrt{a_i}}{\sqrt{a_1}+\sqrt{a_2}}$, and to satisfy the condition of Lemma \ref{lem_random_partition} we need 
\begin{equation*}
a_in\le anp_i^2-\sqrt{anp_i^2+2an^2p_i^3-an(2n+1)p_i^4}.
\end{equation*}
Since $n>a$, it is enough that for $i=1$, $2$,
\begin{equation}\label{eq_rand_par}
a_i \le ap_i^2 - \sqrt{p_i^2+2ap_i^3-2ap_i^4}.
\end{equation}
This holds for large enough $a$, proving the following corollary --- this is the form of the result which we shall need in our proof of Theorem \ref{thm_main}.
\begin{corollary}\label{Cor_good_partition}
 Define a function $m':\mathbb R^+ \times \mathbb R^+ \mapsto \mathbb R^+$ by letting $m'(a_1,a_2)$ be the least real $a$ such that with $p_i = \frac{\sqrt{a_i}}{\sqrt{a_1}+\sqrt{a_2}}$ the equation \eqref{eq_rand_par} holds for $i\in \{1,2\}$. Then for all positive integers $n$,
\[
 m(n,a_1n,a_2n)\le m'(a_1,a_2)n.
\]
\end{corollary}

\subsection{Dominating sets}\label{subsec_dominating}

Next, we quote a bound on the total $k$-domination number of a graph with minimum degree at least $\delta$. For a graph $G=(V,E)$, the \emph{total $k$-domination number} of $G$, $\gamma_{k}^t(G)$, is the cardinality of the smallest vertex set $D\subseteq V$ such that $|N_G(v)\cap D|\ge k$ for each vertex $v\in V$. The following theorem was proved by Henning and Kazemi \cite{HK}: 
\begin{theorem}\label{thm_total_dom}
Suppose $G$ is a graph with minimum degree $\delta\ge k$, and $0\le p \le 1$. Then
\begin{equation*}
\gamma_k^t(G)\le n\left(p+\sum_{i=0}^{k-1}(k-i)\binom{\delta}{i}p^i(1-p)^{\delta-i}\right).
\end{equation*}
\end{theorem}
To sketch the proof of this theorem, for each vertex $v$ we first fix a set $S_v$ of $\delta$ neighbours of $v$. Then, we select a random subset $R$ of the vertices of $G$ by taking each with probability $p$. For each vertex $v$ which has $i < k$ members of $S_v$ in $R$, we add $k-i$ of its neighbours to $R$. The result is a $k$-dominating set, whose expected size is at most the bound in Theorem \ref{thm_total_dom}.

For positive integers $k$ and $\delta$, let $z(k,\delta)$ be the least real number $s$ so that if a graph $G=(V,E)$ has minimum degree at least $\delta$ we have $\gamma_{k}^t(G)\le s|V|$. For fixed $k$ and $\delta$ large, the best bound on $\gamma_k^t(G)$ is given when $p=\frac{\ln \delta}{\delta}(1+o(1))$, which gives a bound on $\gamma_k^t(G) \le \frac{n\ln\delta}{\delta}(1+o(1))$ -- and so $z(k,\delta) \le \frac{\ln\delta}{\delta}(1+o(1))$.

\subsection{Colouring graphs modulo $k$}\label{subsec_k_colourings}
In this subsection we prove four technical lemmas on colouring graphs modulo $k$ for some integer $k$. These lemmas will be important in our proof of Theorem \ref{thm_main}, where we shall often ensure that the induced sums at various vertices of a graph differ modulo $k$. Before we embark on the proofs of these lemmas, we introduce some terminology for colourings of graphs.

Given a graph $G=(V,E)$, an edge subset $E_1\subseteq E$ with a colouring $f:E_1\mapsto \mathbb N$, and a vertex colouring $g:V \mapsto \mathbb N$, we define the \emph{partial sum} of a vertex $v$ to be
\[
 s_{(G,f,g)}(v) = g(v) + \sum_{v\in e \in E_1}f(e).
\]
Now, for a vertex set $S\subseteq V$ and integers $k$ and $i$, we define
\[
 n_{(G,f,g,S,k)}(i)= |\{v\in S: s_{(G,f,g)}(v)\equiv i \pmod{k}\}|.
\]
In both these definitions, if the graph $G$ is clear from context it will be omitted.

The next lemma is a simple result which will allow us to colour a graph $G$ consisting of isolated edges so that the vertex sums $s_{(G,f,g)}(v)$ are not $0$ or $1$ modulo $k$, and don't take any other value modulo $k$ too often. This result will be used to prove Lemma \ref{lem_kcolour}, which is an equivalent lemma for a general graph.
\begin{lemma}\label{lem_iso_kcolour}
Let $k$ be an odd integer with $k\ge 5$. Suppose that $G=(V,E)$ is a graph consisting only of isolated edges. Then for any colouring $g: V \mapsto \mathbb N$, there exists a colouring $f: E\mapsto \{0,\dots,k-1\}$ such that
\begin{enumerate}
\item $n_{(G,f,g,V,k)}(0)=n_{(G,f,g,V,k)}(1)=0$,
\item for each $2\leq i \leq k-1$, $n_{(G,f,g,V,k)}(i)\leq |V|/(k-3) + k + 1$.
\end{enumerate}
\end{lemma}
\begin{proof}
Let the edges of $G$ be $\{e_1,\dots,e_r\}$, and write $e_i$ as $v_{i1}v_{i2}$, such that $g(v_{i1})-g(v_{i2}) \equiv a \pmod k$ for some $0\le a\le (k-1)/2$. Then for $a\in \{0,\dots,(k-1)/2\}$, let $G_a=(V_a,E_a)$ be the graph consisting of those edges of $G$ for which $g(v_{i1})-g(v_{i2})\equiv a \pmod k$. We shall label each $E_a$ separately.

Let $H_a$ be the graph on vertex set $\{0,\dots,k-1\}$ given by joining two integers if they differ by $a$ modulo $k$. In the case $a=0$, we allow $H_a$ to have loops. Then we choose a colouring $f_a: E_a\mapsto \{0,\dots,k-1\}$ by choosing a function $f'_a: E_a\mapsto E(H_a)$. If $e_i\in E_a$ and $f'_a(e_i) = \{u,u+a\}$, we set $f_a(e_i)$ so that $s_{(G_a,f_a,g)}(v_{i1})\equiv u+a \pmod k$, and $s_{(G_a,f_a,g)}(v_{i2})\equiv u \pmod k$. Then $n_{(G_a,f_a,g,V,k)}(i)$ is the number of edges of $E_a$ such that $f'_a(E_a)$ contains $i$.

Let $H_a'$ be the graph $H_a\setminus \{0,1\}$. Since $k$ is odd, the components of $H_a$ are odd cycles (including $1$-cycles if $a=0$), and so the components of $H_a'$ are odd length cycles, paths with an even number of vertices, and at most one path with an odd number of vertices. We pick some (not necessarily distinct) edges of $H_a'$ in each component as follows. For an odd length cycle, we pick every edge once. For a path $v_{1}\dots v_{2k}$ on an even number of vertices, we pick each of the edges $v_{2i-1}v_{2i}$ twice for $1\le i \le k$. For a path $v_{1}\dots v_{2k+1}$ on an odd number of vertices, we again pick each of the edges $v_{2i-1}v_{2i}$ twice for $1\le i \le k$. Then we have picked at least $k-3$ edges of $H_a'$, such that each vertex appears in at most $2$ of them. Let these edges be $e'_1,\dots, e'_{t}$. Then we define the function $f'_a: E_a \mapsto H_a$ to have its image in the set $\{e'_1,\dots,e'_{t}\}$, taking each element in this set at most $|E_a|/t + 1$ times. With $f_a$ defined from $f'_a$ as above, we have
\begin{enumerate}
\item $n_{(G_a,f_a,g,V_a,k)}(0)=n_{(G_a,f_a,g,V_a,k)}(1)=0$,
\item for each $2\leq i \leq k-1$, $n_{(G_a,f_a,g,V_a,k)}(i)\leq 2(|E_a|/t + 1) \le |V_a|/(k-3)+2$.
\end{enumerate}
Our colouring $f$ of $E(G)$ is defined by $f(e)=f_a(e)$ for $e\in E_a$. Then for all $0\le i \le k-1$ the colouring $f$ satisfies
\[
n_{(G,f,g,V,k)}(i) = \sum_{a=0}^{(k-1)/2}n_{(G_a,f_a,g,V_a,k)}(i).
\]
For $i = 0$ or $1$, this sum is zero, and for $2\le i \le k-1$ the sum is at most
\[
\sum_{a=0}^{(k-1)/2}|V_a|/(k-3)+2 = |V|/(k-3) + k+1,
\]
as required.
\end{proof}

Our next lemma concerns colouring the edges of a graph with no isolated vertices to achieve certain values for the $n_{(G,f,g,S,k)}(i)$ on some set $S$ --- this will be used to prove Lemma \ref{lem_k12colour}. The worst case is the one we have already addressed in Lemma \ref{lem_iso_kcolour}, when $G$ consists only of isolated edges and $S=V$.
\begin{lemma}
\label{lem_kcolour}
Let $k$ be an odd integer with $k\ge 5$. Suppose that $G=(V,E)$ is a graph with no isolated vertices, with $S\subseteq V$. Then for any colouring $g: V \mapsto \mathbb N$, there exists a colouring $f: E\mapsto \{0,\dots,k-1\}$ such that
\begin{enumerate}
\item $n_{(f,g,S,k)}(0)=n_{(f,g,S,k)}(1)=0$,
\item for each $2\leq i \leq k-1$, $n_{(f,g,S,k)}(i)\leq |S|/(k-3) + k + 2$.
\end{enumerate}
\end{lemma}
\begin{proof}
Let $G_1$, \dots, $G_m$ be the components of $G$, ordered such that for some $r$ we have $G_1$, \dots, $G_r\subseteq S$, and $G_{r+1}$,\dots, $G_m\not\subseteq S$. For each $1\le i \le r$, $e_i$ be any edge in $E(G_i)$. Let $V' = S\setminus\bigcup_{i=1}^r e_i$. Then we claim that for any function $t:V'\mapsto \mathbb N$ there is a colouring $f': E\setminus \{e_1,\dots,e_r\}\mapsto \{0,\dots,k-1\}$ such that for each $v\in V'$ we have $s_{(f',g)}(v)\equiv t(v)\pmod k$.

Indeed, to construct such a colouring $f'$, it is enough to construct it for each component $G_i$. If $i>r$, let $T$ be any spanning tree of $G_i$, and colour $E(G_i)\setminus T$ arbitrarily. Now, fix a vertex $v_0\in G_i\setminus S$, and colour the edges of $T$ by removing a leaf $v\neq v_0$ from $T$ and colouring the corresponding edge of $T$, such that if $v\in S$ then the total sum $s_{(G,f',g)}(v)$ is equal to $t(v)$ modulo $k$. If $i\le r$, we proceed similarly, but this time we must ensure $e_i\in E(T)$. We now colour $E(T)\setminus \{e_i\}$ by removing leaves $v$ which are not in $e_i$ from $T$, and colouring the corresponding edge of $T$ such that $s_{(G,f',g)}(v) \equiv t(v)$ modulo $k$. 

Using this, we choose $f':E\setminus \{e_1,\dots,e_r\}\mapsto \{0,\dots,k-1\}$ so that the sums $s_{(G,f',g)}(v)$ for $v\in V'$ are not congruent to $0$ or $1$ modulo $k$, and are distributed as evenly as possible among the congruency classes in the set  $\{2,\dots,k-1\}$ modulo $k$. In particular, for any $2\le i\le k$ we have $n_{(G,f',g,V',k)}\le |V'|/(k-2)+1$.

We shall set $f$ to be equal to $f'$ on  $ E\setminus \{e_1,\dots,e_r\}$; it remains to colour the $e_i$ for $1\le i \le r$. We do this using Lemma \ref{lem_iso_kcolour}. Let $G'$ be the graph consisting only of the isolated edges $e_i$, and for $v\in V(G')$ let $g'$ be the function $s_{(G,f',g)}(v)$. We apply Lemma \ref{lem_iso_kcolour} to the graph $G'$, with the vertex colouring $g'$. This guarantees us a colouring $f'':E(G')\mapsto \{0,\dots,k-1\}$ of the edges $e_i$ such that
\begin{enumerate}
\item $n_{(G',f'',g',V(G'),k)}(0)=n_{(G',f'',g',V(G'),k)}(1)=0$,
\item for each $2\leq i \leq k-1$, $n_{(G',f'',g',V(G'),k)}(i)\leq |V(G')|/(k-3) + k + 1$.
\end{enumerate}
We set $f$ to be equal to $f''$ on $E(G')$, and $f'$ otherwise. For a vertex $v\in V(G')$ we have $s_{(G,f,g)}(v) = s_{(G,f',g)}(v)+s_{(G',f'',0)}(v) = s_{(G',f'',g')}(v)$ by the definition of $g'$. For a vertex $v\notin V(G')$, we have $s_{(G,f,g)}(v) = s_{(G,f',g)}(v)$, since $f''$ labels no edge which includes $v$. Hence for $0\le i \le k-1$ we have
\begin{align*}
n_{(G,f,g,S,k)}(i) &= n_{(G,f,g,V(G'),k)}(i) + n_{(G,f,g,V',k)}(i)\\
&= n_{(G',f'',g',V(G'),k)}(i) + n_{(G,f',g,V',k)}(i).
\end{align*}
From our conditions of $f''$ and $f'$, if $i=0$ or $1$ we have $n_{(G,f,g,S,k)}(i) = 0$, and otherwise we have
\[
n_{(G,f,g,S,k)}(i) \le |V(G')|/(k-3) + k + 1 + |V'|/(k-2) + 1 \le |S|/(k-3) + k + 2,
\]
as required.
\end{proof}

This allows us to prove a lemma about labelling the edges of a graph $G$ with a vertex partition $V_1\cup V_2$, so that for $i=1$, $2$ the sums at vertices in $V_i$ are not equal to $0$ or $1$ modulo $k_i$, and there are not too many of these sums in any congruency class modulo $k_i$. This lemma will be needed in Section \ref{sec_form_to_anti}. To prove the lemma, we shall consider a spanning subgraph $H$ of $G$. Starting with a near-arbitrary labelling of $E(G)$, we shall first switch the labels on the edges in $E(H)$ with some labels on edges in $V_2$, to fix sums of vertices in $V_1$ modulo $k_1$. We shall then switch the labels on the edges in $E(H)$ with some labels on edges in $V_1$, to fix sums of vertices in $V_2$ modulo $k_2$, while not affecting our labelling modulo $k_1$. For each of these steps, we shall invoke Lemma \ref{lem_kcolour}.

Given subsets $A$ and $B\subseteq V(G)$, we denote by $E_G(A)$ the set of edges of $G$ contained in $A$, and $E_G(A,B)$ the set of edges of $G$ which can be written $ab$ with $a\in A$ and $b\in B$.

\begin{lemma} \label{lem_k12colour}
Let $k_1$ and $k_2$ be coprime odd integers, both at least $5$, let $G=(V,E)$ be a graph with no isolated vertices, and let $L$ be a set of integers of size $|E|$. Suppose that there exists a partition of $V$ into vertex classes $V_1$ and $V_2$ such that $|E_G(V_1)|\ge (k_1k_2+1)|V|$, and $|E_G(V_2)|\ge (k_1+1)|V|$, and that $L$ contains at least $|V|-1$ labels in each congruency class modulo $k_1k_2$, and at least $(k_2+1)(|V|-1)$ labels in each class modulo $k_1$. Then for any function $g: V\mapsto \mathbb N$ there exists a bijective labelling $f: E \mapsto L$ such that
\begin{enumerate}
 \item\label{it_k1_zero} $n_{(G,f,g,V_1,k_1)}(0)=n_{(G,f,g,V_1,k_1)}(1)=0$,
\item\label{it_k1_nonzero} for each $2\le i\le k_1-1$, $n_{(G,f,g,V_1,k_1)}(i)\le |V_1|/(k_1-3) + k_1 + 2$,
\item\label{it_k2_zero} $n_{(f,g,V_2,k_2)}(0)=n_{(G,f,g,V_2,k_2)}(1)=0$,
\item\label{it_k2_nonzero} for each $2\le i\le k_2-1$, $n_{(G,f,g,V_2,k_2)}(i)\le |V_2|/(k_2-3) + k_2 + 2$.
\end{enumerate}
\end{lemma}
\begin{proof}
First, let $H$ be a minimal spanning subgraph of $G$ with no isolated vertices --- so we have $|E(H)|\le |V|-1$. Now, let $A_1$ be a subset of $E_G(V_1)\setminus E(H)$, and $A_2$ a subset of $E_G(V_2)\setminus E(H)$, containing $k_1k_2|E(H)|$ and $k_1|E(H)|$ edges respectively; these exist because $|E_G(V_1)\setminus E(H)|\ge k_1k_2|V| > k_1k_2 |E(H)|$, and similarly for $V_2$. We label $A_1$ and $A_2$ injectively from $L$ such that for each $i\in \{0,...,k_1k_2-1\}$ there are $|E(H)|$ edges in $A_1$ with labels congruent to $i$ modulo $k_1k_2$, and for each $i\in \{0,...,k_1-1\}$ there are $|E(H)|$ edges in $A_2$ with labels congruent to $i$ modulo $k_1$. There are enough labels of $L$ in each congruency class to do this by our restrictions on $L$. Next, we assign the other labels in $L$ injectively but otherwise arbitrarily to $E \setminus (A_1\cup A_2)$ --- let the resulting bijective labelling from $E$ to $L$ be $f_2$.

Now, we define $g_2: V\mapsto \mathbb N$ by $g_2(v)=s_{(G,f_2,g)}(v)-\sum_{v\in e \in E(H)} f_2(e)$ --- that is, $s_{(G,f_2,g)}(v)$, but ignoring the labels of edges in $H$. Applying Lemma \ref{lem_kcolour} to the graph $H$, with $S=V_1$, $k=k_1$, and $g=g_2$ gives us a colouring $f': E(H)\mapsto \{0,...,k_1-1\}$ such that
\begin{enumerate}
\item $n_{(H,f',g_2,V_1,k_1)}(0)=n_{(H,f',g_2,V_1,k_1)}(1)=0$,
\item for each $2\le i \le k_1-1$, $n_{(H,f',g_2,V_1,k_1)}(i)\le |V_1|/(k_1-3) + k_1 + 2$.
\end{enumerate}
We use this colouring $f'$ to define a new bijective labelling $f_1: E\mapsto L$ as follows. For every edge $e\in E(H)$, we choose an edge $a(e)\in A_2$ such that $f_2(a(e))\equiv f'(e)$ (mod $k_1$). We choose the $a(e)$ to be distinct --- this is possible, since for each $i\in \{0,\dots,k_1-1\}$ there are $|E(H)|$ edges $e'\in A_2$ with $f_2(e')\equiv i \pmod{k_1}$. Now, for each $e\in E(H)$, we set $f_1(e)=f_2(a(e))$, and $f_1(a(e))=f_2(e)$, and for edges not in $E(H)$ or the image of $a$ we set $f_2=f_1$.

To construct this colouring from $f_2$, we have taken some pairs of edges, with no edge appearing in two pairs, and swapped the labels on each pair. Hence the labels used by $f_1$ are exactly the same as those used by $f_2$, and so $f_1$ is a bijective labelling from $E\mapsto L$. By our choice of $g_2$, $s_{(G,f_1,g)}(v)\equiv s_{(H,f',g_2)}(v) \pmod{k_1}$ for each $v\in V_1$, and so $f_1$ satisfies Conditions \ref{it_k1_zero} and \ref{it_k1_nonzero} of the lemma.

We proceed similarly to change the sums at vertices of $V_2$ modulo $k_2$ --- but this time we shall also ensure we do not change the labelling modulo $k_1$. We define $g_1: V\mapsto \mathbb N$ by $g_1(v)=s_{(G,f_1,g)}(v)-\sum_{v\in e \in E(H)} f_1(e)$. Applying Lemma \ref{lem_kcolour} to the graph $H$, with $S=V_2$, $k=k_2$, and $g=g_1$ gives us a colouring $f'': E(H)\mapsto \{0,...,k_2-1\}$ such that
\begin{enumerate}
\item $n_{(H,f'',g_1,V_2,k_2)}(0)=n_{(H,f'',g_1,V_2,k_2)}(1)=0$,
\item for each $1\le i \le k_2-1$, $n_{(H,f'',g_1,V_2,k_2)}(i)\le |V_2|/(k_2-1) + k_2 + 2$.
\end{enumerate}
We use this colouring $f''$ to define a new bijective labelling $f: E\mapsto L$ as follows. For every edge $e\in E(H)$, we choose an edge $a(e)\in A_1$ such that $f_1(a(e))\equiv f''(e)$ (mod $k_2$), but now we also insist that $f_1(a(e)) \equiv  f_1(e)$ (mod $k_1$). We choose the $a(e)$ to be distinct --- this is possible, since for each $i\in \{0,\dots,k_1k_2-1\}$ there are $|H|$ edges $e'\in A_1$ with $f_1(e')\equiv i \pmod{k_1k_2}$. Now, for each $e\in E(H)$, we set $f(e)=f_1(a(e))$, and $f(a(e))=f_1(e)$, and for edges not in $E(H)$ or the image of $a$ we set $f=f_1$.

As before, to construct $f$ from $f_1$, we have taken some pairs of edges and swapped the labels on each pair, and again no edge appears in two pairs. Hence the labels used by $f$ are exactly the same as those used by $f_1$, and so $f$ is also a bijective labelling from $E\mapsto L$. By our choice of $g_1$, $s_{(G,f,g)}(v)\equiv s_{(H,f'',g_1)}(v) \pmod{k_2}$ for each $v\in V_2$, and so $f$ satisfies Conditions \ref{it_k2_zero} and \ref{it_k2_nonzero} of the lemma. Since the labellings $f_1$ and $f$ are identical viewed modulo $k_1$, $f$ also satisfies Conditions \ref{it_k1_zero} and \ref{it_k1_nonzero}.
\end{proof}

The final lemma of the section is another simple technical lemma, which concerns labelling a graph with a vertex partition; this lemma will be used in Section \ref{sec_form_to_anti}. The proof is similar in style to that of Lemma \ref{lem_kcolour}.

\begin{lemma}\label{lem_AB_colour}
Let $k_1$ and $k_2$ be integers with $k_1\ge 5$. Let $G=(V,E)$ be a graph with a vertex partition into vertex sets $A$ and $B$, and let $B'$ be a set of vertices contained in $B$. Suppose that every vertex in $A$ has at least two edges to vertices in $B$, and that every vertex in $B'$ has at least one edge to a vertex in $A$. Suppose further that $L$ is a set of at least $|E|+k_1k_2(2|A|+|B'|)$ integers, containing at least $2|A|+|B'|$ representatives of each congruency class modulo $k_1k_2$. Then for any functions $g:V\mapsto \mathbb{N}$ and $t:A\mapsto \mathbb{N}$, there is an injective labelling $f: E\mapsto L$ such that
\begin{enumerate}
 \item for each $v\in A$, $s_{(f,g)}(v)\equiv t(v) \pmod {k_1k_2}$,
 \item $n_{(f,g,B',k_1)}(0)=n_{(f,g,B',k_1)}(1)=0$,
\item for each $2\le i\le k_1-1$, $n_{(f,g,B',k_1)}(i)\le |B'|/(k_1-4)+2k_1-3$.
\end{enumerate}
\end{lemma}
\begin{proof}
 Let $E'\subseteq E$ be a set of edges which contains at least $2$ edges to $B$ from every vertex of $A$, and at least $1$ edge to $A$ from every vertex of $B'$, and has at most $2|A|+|B'|$ edges. Let $G'$ be the graph $(V,E')$. Also, let $L' \subseteq L$ be a set of $k_1k_2(2|A|+|B'|)$ labels, with $2|A|+|B'|$ labels in each congruency class modulo $k_1k_2$. Let $f':E\setminus E'\mapsto L\setminus L'$ be an arbitrary injective mapping; such a mapping exists, as $|L|\ge  |E| + |L'|$. For $v\in V$, let $g'(v)= s_{(G,f',g)}(v)$.

Now, $L'$ contains $2|A|+|B'|$ labels in each congruency class modulo $k_1k_2$, whereas $E'$ contains at most $2|A|+|B'|$ edges, so we can label $E'$ however we wish modulo $k_1k_2$ while labelling injectively from $L'$. Hence to label $E'$ we first define a colouring $f_k:E'\mapsto \{0,\dots,k_1k_2-1\}$, and then assign labels of $L'$ injectively to agree with $f_k$ modulo $k_1k_2$. Let the components of $G'$ be $G'_1$, \dots, $G'_m$, ordered such that for some $r$ we have $G'_1$, \dots, $G'_r\subseteq A\cup B'$, and $G'_{r+1}$,\dots, $G'_m\not\subseteq A\cup B'$. For each $1\le i \le r$, select a vertex $a_i\in A\cap G'_i$, and two vertices $b_{i1}$ and $b_{i2}\in B'\cap G'_i$ which are neighbours of $A$ in $G'$. Let 
\[
V'= (A\cup B') \setminus \left(\bigcup_{1\le i \le r}\{a_i,b_{i1},b_{i2}\}\right).
\]
Then we can colour $E' \setminus \bigcup_{1\le i \le r} \{a_ib_{i1},a_ib_{i2}\}$ such that the vertices in $V'$ receive any specified sums modulo $k_1k_2$. We do this similarly to in the proof of Lemma \ref{lem_kcolour}; we take a spanning tree $T_i$ for each component $G'_i$ of $G'$, and when $i\le r$ ensure that the path $b_{i1}a_ib_{i2}$ is contained in $T_i$. We colour the edges not in some $T_i$ arbitrarily, and then remove leaves which are not in $\{a_i,b_{i1},b_{i2}\}$ from $T_i$, colouring the corresponding edges so that every vertex $v\in V'$ receives the desired sum modulo $k$.

Using this, we choose a colouring $f_k':E' \setminus \bigcup_{1\le i \le r} \{a_ib_{i1},a_ib_{i2}\}\mapsto \{0,\dots,k_1k_2-1\}$ such that vertices $v\in V'\cap A$ receive sum congruent to $t(v)$ modulo $k_1k_2$, and vertices $v\in V'\cap B'$ receive sums which are not congruent to $0$ or $1$ modulo $k_1$, and are split as evenly as possible between the congruency classes in the set $\{2,\dots,k_1-1\}$ modulo $k_1$. For our final colouring $f_k$, we shall take $f_k=f_k'$ on the domain of $f_k'$.

At this stage, the uncoloured edges consist of $r$ independent copies of $P_3$, $b_{i1}a_ib_{i2}$, with $a_i \in A$ and $b_{i1}$, $b_{i2} \in B'$. However we colour the edges $b_{i1}a_i$ and $b_{i2}a_i$, we shall have
\[
s_{(f_k,g')}(b_{i1})+s_{(f_k,g')}(b_{i2})- 2s_{(f_k,g')}(a_i) = s_{(f_k',g')}(b_{i1}) + s_{(f_k',g')}(b_{i2}) - 2s_{(f_k',g')}(a_i),
\]
and so the constraint that $s_{(f_k,g')}(a_i)$ is congruent to $t(a_i)$ modulo $k_1k_2$ leads to a constraint of the form $s_{(f_k,g')}(b_{i1})+s_{(f_k,g')}(b_{i2})\equiv m_i \pmod {k_1k_2}$ for some $m_i\in \{0,\dots,k_1k_2-1\}$. For each $j \in \{0,\dots,k_1-1\}$, let $I_j\subseteq \{1,\dots,r\}$ be the set of those $i$ with $m_i \equiv j\pmod{k_1}$, $A'_j$ be  $\{a_i:i\in I_j\}$, $B'_j$ be $\bigcup_{i\in I_j}\{b_{i1},b_{i2}\}$, and $E'_j$ be $\bigcup_{i\in I_j} \{a_ib_{i1}, a_ib_{i2}\}$. We shall colour the edge sets $E'_j$ independently of each other.

Writing $I_j = \{i_1,\dots,i_s\}$, we wish to pick the colours of the edges $a_{i_\ell}b_{i_\ell 1}$ and $a_{i_\ell}b_{i_\ell 2}$. Let $\ell \equiv c \pmod {k_1-4}$, where $c\in \{1,\dots,k_1-4\}$. Then we colour $a_{i_\ell}b_{i_\ell 1}$ and $a_{i_\ell}b_{i_\ell 2}$ so that $s_{(f_k,g')}(a_{i_\ell})\equiv t(a_{i_\ell}) \pmod {k_1k_2}$, and $s_{(f_k,g')}(b_{i_\ell 1})\equiv d \pmod {k_1}$, where $d$ is the $c^{th}$ element of the set $\{0,\dots,k_1-1\}\setminus \{0,1,j,j-1\}$.

With this colouring, we have $s_{(f_k,g')}(b_{i_\ell 2})\equiv j-d \pmod {k_1}$, and in particular $s_{(f_k,g')}(b_{i_\ell 2})$ is not congruent to $0$ or $1$ modulo $k_1$. Also, for $k-4$ consecutive members of $I_j$, we have $k-4$ pairs $(b_{i_\ell 1},b_{i_\ell 2})$. Of these, exactly one of the $b_{i_\ell1}$ has $s_{(f_k,g')}(b_{i_\ell 1})$ congruent to each element of $\{0,\dots,k-1\}\setminus \{0,1,j,j-1\}$ modulo $k$, and the same holds for the $b_{i_\ell 2}$. Hence the $n_{(f_k,g',B'_j,k_1)}(i)$ satisfy
\begin{enumerate}
 \item $n_{(f_k,g',B'_j,k_1)}(i) = 0 \text{ for all } i\in \{0,1,j-1,j\}$,
 \item $n_{(f_k,g',B'_j,k_1)}(i) \le |B'_j|/(k_1-4)+2 \text{ for all }  1\le i \le k_1,\,i\notin\{0,1,j-1,j\}$.
\end{enumerate}
Then for any $2\le a \le k_1-1$, $n_{(f_k,g',B',k_1)}(a)$ is at most
\[
 |V'\cap B|/(k_1-2)+1 + \sum_{\substack{0\le j \le k_1-1,\\ j\notin\{a,a+1\}}} (|B'_j|/(k_1-4)+2)\le |B'|/(k_1-4) + 2k_1 -3.
\]
Taking any injective labelling $f:E\mapsto L$ which is equal to $f'$ on $E\setminus E'$ and agrees with $f_k$ modulo $k_1k_2$ on $E'$, $f$ satisfies the conditions of the lemma; indeed for all $v$ in $A\cup B'$ we have $s_{(G,f,g)}(v)\equiv s_{(G,f_k,g')}\pmod {k_1k_2}$, so the properties we require for $f$ follow from those we have proved for $f_k$.
\end{proof}

\section{Reduction to a minimum degree problem}\label{sec_av_to_min}
In this section, our aim is to reduce the problem of producing an antimagic labelling for a graph with large average degree to a similar problem for a graph with large minimum degree. To do this, we must first recall the notion of the $r$-core of a graph. The \emph{$r$-core} of a graph $G=(V,E)$, which we denote $V_{c_r}$, is the largest set of vertices such that every vertex $v\in V_{c_r}$ has $|E_G(\{v\},V_{c_r})|\ge r$. The $r$-core of $G$ can be obtained by successively removing vertices of $G$ with degree at most $r-1$; this shows that the subgraph of $G$ induced by $r$-core of $G$ contains all but at most $(r-1)|V\setminus V_{c_r}|$ of the edges of $G$.

To label a graph $G=(V,E)$ with large average degree, we shall pick appropriate integers $\delta$ and $k$. Defining $V_1$ to be the $\delta$-core of $G$, and $V_0=V\setminus V_1$, we shall first label $E_G(V_0,V)$, so that the sums at vertices of $V_0$ are all divisible by $k$, and none are equal. We shall then label $E_G(V_1)$ so that the sums at vertices of $V_1$ are not divisible by $k$, and none are equal --- this gives us our antimagic colouring. For the first stage, we need the following lemma:
\begin{lemma}\label{lem_V0_0}
Let $k$ be an odd positive integer, let $G=(V,E)$ be a graph, and let $V_0$ and $V_1$ be vertex sets partitioning $V$. Then there is a colouring $f:E_G(V_0,V)\mapsto \{0,\dots,k-1\}$ so that  $s_{(f,0)}(v)$ is divisible by $k$ for every vertex $v\in V_0$, and each colour is used at most $|E_G(V_0,V)|/k+|V_0|$ times.
\end{lemma}
\begin{proof}
It is enough to prove the lemma for a connected graph; indeed, for a general graph we can simply apply the lemma to each connected component. We split the proof into three cases. Firstly, if $V_1$ is non-empty, let $F$ be a forest with edges in $E_G(V_0,V)$ which spans $V_0$, and has exactly one vertex of $V_1$ in each component. We colour $E_G(V_0,V)\setminus E(F)$ as evenly as possible with $\{0,\dots,k-1\}$, and otherwise arbitrarily. Then there is some colouring of $F$ such that the overall sum at every vertex in $V_0$ is divisible by $k$; we can obtain such a colouring by succesively removing leaves $v$ of $F$ which are in $V_0$, and colouring the corresponding edge to ensure the sum at $v$ is divisible by $k$.

Secondly, if $V_1$ is empty and $G$ is not bipartite, let $G'$ be any connected subgraph of $G$ which spans $V$ and has exactly one cycle $C$, which is of odd length --- so $G'$ has $|V|$ edges. We colour $E\setminus E(G')$ as evenly as possible with $\{0,\dots,k-1\}$, and otherwise arbitrarily. Then we claim that there is some colouring of $E(G')$ so that the overall sum at every vertex in $V$ is divisible by $k$. We obtain this colouring by succesively removing degree $1$ vertices $v$ from $G'$, colouring the corresponding edges to ensure the overall sum at $v$ is divisible by $k$. We do this until we are left with only the odd cycle $C$; we now need to colour the edges of $C$ so that every vertex on it has sum divisible by $k$. If the cycle is of length $r$, this is equivalent to solving a system of equations
\begin{align*}
a_1+a_2 &\equiv b_1 \pmod k \\
a_2+a_3 &\equiv b_2 \pmod k \\
&\dots \\
a_r+a_1 &\equiv b_r \pmod k. \\
\end{align*}
Here, the $a_i$ correspond to the colours being given to the edges of the cycle $C$, and the $b_i$ to the remaining sum needed at the vertices of $C$ to bring the sum to $0$ modulo $k$. Since $r$ and $k$ are both odd, this system of equations does indeed have a solution.

Finally, suppose $V_1$ is empty and $G$ is bipartite, with vertex classes $A$ and $B$. Let $T$ be a spanning tree of $G$, and $v_0$ any vertex of $G$ --- say $v_0\in A$. We colour $E\setminus E(T)$ as evenly as possible with $\{0,\dots,k-1\}$, and otherwise arbitrarily. Then there is some colouring of $T$ such that every vertex in $V$ other than $v_0$ has induced sum divisible by $k$; as ever, we obtain such a colouring by succesively removing leaves of $T$ and colouring the corresponding edge. Let $f$ be the colouring $f:E \mapsto \{0,\dots,k-1\}$ this gives. Since $\sum_{v\in A} s_{(f,0)}(v)=\sum_{v\in B} s_{(f,0)}(v)=\sum_{e\in E}f(e)$, the induced sum at $v_0$ is also divisible by $k$.

In each case, we have coloured all but at most $|V_0|$ of the edges of $E_G(V_0,V)$ as evenly as possible with $\{0,\dots,k-1\}$, and then coloured the remainder in some specified way. Hence each of the colours $\{0,\dots,k-1\}$ is used on at most $|E_G(V_0,V)|/k+|V_0|$ edges.
\end{proof}

Now we are in a position to prove a lemma which allows us to label $E_G(V_0,V)$ so that the sums at vertices in $V_0$ are distinct, and all are divisible by some integer $k$.
\begin{lemma}\label{lem_V0_labelling}
Let $\delta$ and $k$ be odd positive integers, and let $G=(V,E)$ be a graph with no isolated edges and at most one isolated vertex, with $\delta$-core $V_1$ and $V_0=V\setminus V_1$. Let $L$ be an interval of $\mathbb N$ of length at least $(\delta-1+3k)|V_0|$. Then there is an injective labelling $f: E(V_0,V)\mapsto L$ such that the sum $s_{(f,0)}(v)$ is divisible by $k$ for each $v\in V_0$, and $s_{(f,0)}(v_1)\neq s_{(f,0)}(v_2)$ for distinct vertices $v_1$ and $v_2$ in $V_0$.
\end{lemma}
\begin{proof}
By Lemma \ref{lem_V0_0} we can choose a colouring $f_k: E_G(V_0,V)\mapsto \{0,\dots,k-1\}$ such that every vertex $v\in V_0$ has $s_{(f_k,0)}(v)$ divisible by $k$, and each colour is used at most $|E_G(V_0,V)|/k+|V_0|$ times. Now, we label $E_G(V_0,V)$ with labels from $L$, stepping through the edges in any order. Let $E_G(V_0,V)=\{e_1,\dots,e_r\}$. For $0\le i \le r$ we define a labelling $f^i: \{e_1,\dots,e_i\}\mapsto L$, by setting $f^i=f^{i-1}$ on $\{e_1,\dots,e_{i-1}\}$, and setting $f^i(e_i)=l$ for some label $l$ that obeys the following conditions:
\begin{enumerate}
\item $l$ is in $L$, and $l\equiv f_k(e_i)\pmod{k}$,
\item $l$ is not in the image of $f^{i-1}$,
\item \label{item_no_clash} if $v\in V_0$ and $v\in e_i$, $s_{(f^{i-1},0)}(v)+l \neq s_{(f^{i-1},0)}(v')$ for any $v' \in V_0$ with $v'\notin e_i$.
\end{enumerate}
We claim that there is always a label which obeys these restrictions. Indeed, for $0\le i \le k-1$ let $L_i$ be the set of labels in $L$ which are congruent to $i$ modulo $k$. We wish to label $e_i$ with a label in $L_{f_k(e_i)}$. Since $f_k$ uses each colour at most $|E_G(V_0,V)|/k+|V_0|$ times, the second condition rules out at most $|E_G(V_0,V)|/k+|V_0| - 1$ labels in $L_{f_k(e_i)}$. The third condition applies to at most $2$ distinct vertices $v$, and for each rules out at most $|V_0|-2$ labels. Hence the total number of labels in $L_{f_k(e_i)}$ which violate one of these two conditions is at most
\[
 |E_G(V_0,V)|/k+|V_0| -1 + 2(|V_0|-2) < |E_G(V_0,V)|/k+3|V_0|-1.
\]
On the other hand, since $V_0$ is the complement of the $\delta$-core of $G$, $|E_G(V_0,V)|\le (\delta-1)|V_0|$. Hence $|L|\ge |E_G(V_0,V)|+3k|V_0|$, and so $|L_{f_k(e)}|\ge |E_G(V_0,V)|/k+3|V_0|-1$. So there is some label $l$ which we can use at $e$.

Let the labelling this process gives be $f=f^r$. Since $f$ agrees with $f_k$ modulo $k$, every sum $s_{(f,0)}(v)$ is divisible by $k$ for $v\in V_0$. For $v_1\neq v_2$ vertices of $V_0$, let $e_j$ be the last edge incident with exactly one of $v_1$ and $v_2$ to be labelled; such an edge exists since $G$ has at most one isolated vertex and no isolated edges. When $e_j$ is labelled, Condition \ref{item_no_clash} on $f^j(e_j)$ guarantees $s_{(f^j,0)}(v_1)\neq s_{(f^j,0)}(v_2)$, and hence $s_{(f,0)}(v_1)\neq s_{(f,0)}(v_2)$. \end{proof}

This enables us to give a lemma which is sufficient for graphs with large average degree to be antimagic; Sections \ref{sec_min_to_form} and \ref{sec_form_to_anti} will be devoted to the proof of this lemma. Given a graph $G=(V,E)$ and a function $g:V\mapsto \mathbb N$, we call a colouring $f:E\mapsto \mathbb N$ \emph{$g$-antimagic} if $s_{(f,g)}(v_1)\neq s_{(f,g)}(v_2)$ for distinct vertices $v_1$ and $v_2$ in $V$.

\begin{lemma}\label{lem_min_suffices}
 Let $k_1$ and $k_2$ be sufficiently large odd coprime integers. Then there are constants $c=c(k_1,k_2)$ and $\delta= \delta(k_1,k_2)$ such that if $G=(V,E)$ is a graph with minimum degree at least $\delta$, $L$ is a set of integers of size $|E|$ containing $\{1,\dots,c|V|\}$, and $g$ is a function $g:V\mapsto \mathbb N$, then there exists a $g$-antimagic bijective labelling $f:E\mapsto L$ such that no vertex in $V$ has induced sum $s_{(f,g)}(v)$ divisible by $k_1k_2$.
\end{lemma}
In fact, the truth of this lemma for a single pair of integers $k_1$ and $k_2$ is sufficent to prove Theorem \ref{thm_main}. For integers $a$ and $b$, we define $[a,b]$ to be the set $\{n\in \mathbb N: a\le n \le b\}$.
\subsection{Proof of Theorem \ref{thm_main} from Lemma \ref{lem_min_suffices}}
Suppose that Lemma \ref{lem_min_suffices} holds for some $k_1$ and $k_2$, with constants $c=c(k_1, k_2)$ and $\delta = \delta(k_1,k_2)$. Then we claim that Theorem \ref{thm_main} holds for
\begin{equation}\label{eq_d_suff}
d_0=2\max(c,\delta -1 + 3k_1k_2).
\end{equation}
Indeed, given a graph $G$ with average degree at least $d_0$, let $V_1$ be the $\delta$-core of $G$ and $V_0 = V\setminus V_1$. Now, apply Lemma \ref{lem_V0_labelling} to the graph $G$, with $k=k_1k_2$ and the label set $L'=\left[|E|-(\delta-1+3k_1k_1)|V_0|+1,|E|\right]$. This gives us an injective labelling $f_1:E(V_0,V)\mapsto L'$, so that $s_{(G,f_1,0)}(v) \equiv 0 \pmod {k_1k_2}$ for each $v\in V_0$, and $s_{(G,f_1,0)}(v_1)\neq s_{(G,f_1,0)}(v_2)$ for $v_1$ and $v_2$ distinct vertices in $V_0$. We define $L=[1,|E|]\setminus f_1(E(V_0,V))$; so certainly $[1,nd_0/2-(\delta-1+3k_1k_2)|V_0|]\subseteq L$. Also, note that
\begin{align*}
 nd_0/2 &\ge n\max(\delta-1+3k_1k_2,c)\\
&\ge (\delta-1+3k_1k_2)|V_0| + c|V_1|.
\end{align*}
Hence $L$ contains $[1,c|V_1|]$, and we can apply Lemma \ref{lem_min_suffices} to the integers $k_1$ and $k_2$, the graph $G'=(V_1,E_G(V_1))$ and the label set $L$. The function $g$ we use is $g(v)=s_{(G,f_1,0)}(v)$ for $v\in V_1$. So from the conclusion of Lemma \ref{lem_min_suffices}, there exists a $g$-antimagic bijective labelling $f_2: E_G(V_1)\mapsto L$, so that no vertex in $V_1$ has $s_{(G',f_2,g)}(v)$ divisible by $k_1k_2$. We define the labelling $f:E\mapsto [1,|E|]$ to be equal to $f_1$ on $E_G(V_0,V)$ and equal to $f_2$ on $E_G(V_1)$. Note that $f$ is a bijective labelling from $E$ to $[1,|E|]$; indeed, $f_1$ is bijective from $E(V_0,V)\mapsto f_1(E(V_0,V))$, and $f_2$ is bijective from $E_G(V_1)\mapsto L = [1,|E|]\setminus f_1(E(V_0,V))$.

Now, for $v\in V_1$ we have 
\[
s_{(G,f,0)}(v)= s_{(G',f_2,0)}(v)+s_{(G,f_1,0)}(v)=s_{(G',f_2,g)}(v),
\]
so the sums $s_{(G,f,0)}(v)$ for $v\in V_1$ are distinct and not divisible by $k_1k_2$. Also, $v\in V_0$ we have $s_{(G,f,0)}(v) = s_{(G,f_1,0)}(v)$, as $f_2$ labels no edge incident with $v$, so the sums $s_{(G,f,0)}(v)$ for $v\in V_0$ are distinct and divisible by $k_1k_2$. Hence $s_{(G,f,0)}(v_1)\neq s_{(G,f,0)}(v_2)$ for $v_1$ and $v_2$ distinct vertices in $V$, and so $G$ is antimagic.

\section{A partition of a graph with large minimum degree}\label{sec_min_to_form}
Now we shall begin the proof of Lemma \ref{lem_min_suffices}. Given a graph $G=(V,E)$ with large minimum degree, Lemma \ref{lem_stars} shows that we can pick some vertex disjoint stars in $G$ with a large edge set, such that removing these stars leaves many edges in $G$. We use this to prove Lemma \ref{lem_partition}, which guarantees we can partition $V$ and $E$ into a certain form. In Section \ref{sec_form_to_anti}, we shall show that a graph in this form satisfies the conclusions of Lemma \ref{lem_min_suffices}.

We define a \emph{star} to be a graph $S=(V,E)$ on at least $2$ vertices with a distinguished vertex $c$ such that $E=\{cv: v \in V\setminus \{c\}\}$. The vertex $c$ is called the \emph{centre} of $S$. A \emph{star forest} is just  a collection of vertex-disjoint stars. Also, recall that in Subsection \ref{subsec_dominating}, $z(k,\delta)$ was defined as the smallest real number $s$ such that any graph with minimum degree $\delta$ and $n$ vertices has a set of at most $sn$ vertices with at least $k$ edges to every vertex of $G$.

\begin{lemma}\label{lem_stars}
Let $\delta$, $n$ and $r$ be positive integers with $r\le n\delta/2$ and $\delta \ge 5$. Let $G=(V,E)$ be a graph with minimum degree at least $\delta$ and $|V|=n$. Then there exists a star forest $F_S\subseteq G$ such that the following hold:
\begin{enumerate}
\item \label{Con1_V0_r_edges}$|E_G(V\setminus V(F_S))|\ge r$,
\item \label{Con1_star_edges}$|E(F_S)| \ge n(1/2-z(5,\delta)-2/\delta) - 1 - r/\delta$,
\item \label{Con1_centre_set}There is a set $V_1$ consisting of some of the centres of the stars in $F_S$, such that:
\begin{itemize}
\item Every vertex in $V_1$ has at least $5$ edges to $V\setminus V(F_S)$.
\item Every vertex in $G$ has at least $5$ edges to $(V\setminus V(F_S)) \cup V_1$.
\end{itemize}
\end{enumerate}
\end{lemma}
\begin{proof}
Note that if the lemma holds for a graph $G$, and $G$ is a subgraph of a graph $G'$ on the same vertex set, the lemma also holds for $G'$; indeed, any choice of $F_S$ and $V_1$ which verifies the lemma for $G$ also do so for $G'$. So it is sufficient to prove the lemma for graphs $G$ with minimum degree $\delta$ and no edges $v_1v_2$ for which $v_1$ and $v_2$ have degree greater than $\delta$, since every graph with minimum degree at least $\delta$ has a subgraph satisfying this condition; indeed, such a subgraph can be obtained by successively removing edges between two vertices of degree greater than $\delta$.

Given a graph $G=(V,E)$ of this form, let $V_s$ be the set of vertices of degree $\delta$, and let $V_b= V\setminus V_s$. Note that $E_G(V_b)=\emptyset$. Now, let $D$ be a smallest set of vertices in $G$ with $|N_G(V)\cap D|\ge 5$ for each $v\in V$; since $\delta \ge 5$, $D$ certainly exists. By the definition of $z(k,l)$, $|D|\le z(5,\delta)n$. Let $D_s = D\cap V_s$. We now give an algorithm for choosing our star forest $F_S$, as follows: 
\begin{enumerate}
\item $V(0)=V\setminus D_s$.
\item\label{item_V(t)} Given $V(t)$, let $G(t)$ be the graph $(V(t), E_G(V(t)))$, the graph induced by $G$ on vertex set $V(t)$. Let $V_s(t)= V(t)\cap V_s$, and $V_b(t)=V(t)\cap V_b$.
\item If $|E_G(V(t))| < r + n + \delta$, we terminate the algorithm, setting $F_S$ to be the stars $\{S_1, \dots, S_t\}$ picked so far.
\item Otherwise, we pick a star $S_{t+1}\subseteq G(t)$ with centre $c_{t+1}$ to add to our star forest. To pick the centre $c_{t+1}$ of the star, if there is an edge $v_1v_2\in E_G(V(t))$ with $v_1\in V_s(t)$ and $v_2\in V_b(t)$, we choose any such edge and let $c_{t+1}=v_2$. Otherwise, since $E_G(V_b)=\emptyset$, $E(G(t))$ must contain an edge $v_1v_2$ with both $v_1$ and $v_2$ in $V_s(t)$; then we choose any such edge and let $c_{t+1}=v_1$.
\item\label{item_nhd} To pick the vertex set $A_{t+1}$ for $S_{t+1}$, let $N=N_{G(t)}(c_{t+1})$ be the neighbourhood of $c_{t+1}$ in $G(t)$, and write $N=\{n_1,\dots,n_l\}$. Letting $l'$ be maximal such that $|E_G(V(t)\setminus \{c_{t+1}, n_1,\dots,n_{l'}\})|\ge r$, we set $A_{t+1}=\{c_{t+1}\}\cup \{n_1,\dots,n_{l'}\}$.
\item We set $V(t+1)=V(t)\setminus A_{t+1}$, and go to Step \ref{item_V(t)}.
\end{enumerate}

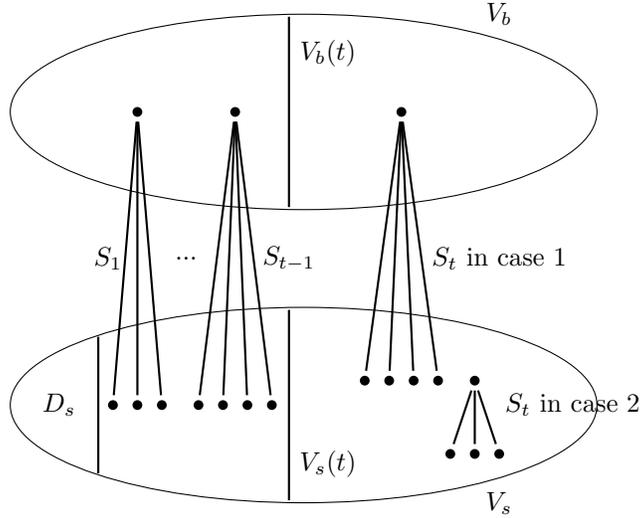
\begin{figure}[ht]
\centering{
\begin{tikzpicture}[scale=0.65,thick]
  \begin{pgfonlayer}{nodelayer}
   \node (a) at (4,9){};
   \node (b) at (4,3){};
   \node (c) at (4.5,10.2){\(V_b(t)\)};
   \node (d) at (4.5,1.8){\(V_s(t)\)};
   \node (e) at (8,11){\(V_b\)};
   \node (f) at (8,1){\(V_s\)};
   \node (g) at (0,6){\(S_1\)};
   \node (h) at (1.6,6){...};
   \node (i) at (3.7,6){\(S_{t-1}\)};
   \node (j) at (-0.2,4.6){};
   \node (k) at (-0.2,1.4){};
   \node (l) at (-1,3){\(D_s\)};
   \node (m) at (3.7,11.15){};
   \node (n) at (3.7,6.85){};
   \node (o) at (3.7,5.15){};
   \node (p) at (3.7,0.85){};
   \node (q) at (8,6){\(S_t\) in case 1};
   \node (r) at (9.5,3){\(S_t\) in case 2};
   \node (c_0) at (0.6,9){};
   \node (c_1) at (2.6,9){};
   \node (c_2) at (6,9){};
   \node (c_3) at (7.5,3.5){};
   
   \foreach \x in {0,...,2}{\node (l0_\x) at ($(0.1,3)+(0.5*\x,0)$){};}
   \foreach \x in {0,...,3}{\node (l1_\x) at ($(1.85,3)+(0.5*\x,0)$){};}
   \foreach \x in {0,...,3}{\node (l2_\x) at ($(5.25,3.5)+(0.5*\x,0)$){};}
   \foreach \x in {0,...,2}{\node (l3_\x) at ($(7,2)+(0.5*\x,0)$){};}

   \begin{scope}
   \draw (a) ellipse (6cm and 2cm);
   \draw (b) ellipse (6cm and 2cm);
   \end{scope}
\end{pgfonlayer}

\foreach \x in {0,...,3}{\fill (c_\x) circle (0.1);}
\foreach \x in {0,...,2}{\fill (l0_\x) circle (0.1);}
\foreach \x in {0,...,3}{\fill (l1_\x) circle (0.1);}
\foreach \x in {0,...,3}{\fill (l2_\x) circle (0.1);}
\foreach \x in {0,...,2}{\fill (l3_\x) circle (0.1);}

\draw (j) to (k);
\draw (m) to (n);
\draw (o) to (p);
\foreach \x in {0,...,2}{\draw (l0_\x) to (c_0);}
\foreach \x in {0,...,3}{\draw (l1_\x) to (c_1);}
\foreach \x in {0,...,3}{\draw (l2_\x) to (c_2);}
\foreach \x in {0,...,2}{\draw (l3_\x) to (c_3);}

\end{tikzpicture}}
\caption{Picking the star $S_t$}
\label{fig_star_struct}
\end{figure}

This algorithm is illustrated by Figure \ref{fig_star_struct}. Now, let $F_S$ be the star forest defined by this algorithm, and let $V_1 = V(F_S) \cap V_b$. We claim that $F_S$ and $V_1$ satisfy the conclusions of the lemma. First, in Step \ref{item_nhd} of the algorithm note that since $E_G(V_b)=\emptyset$ we have $N\subseteq V_s$. So the only vertices of $F_S$ which can lie in $V_b$ are the centres of stars, and so we do indeed have $V_1$ being a subset of the centres of stars in $F_S$.

Also, since in Step \ref{item_nhd} of the algorithm we have $N\subseteq V_s$, all the vertices of $N$ have degree at most $\delta$ in $G(t)$. Thus for $1\le i \le l$ we have
\[
|E_G(V(t)\setminus \{c, n_1, \dots, n_i\})|\ge |E_G(V(t))|-n-i\delta,
\]
and in particular $|E_G(V(t))\setminus \{c, n_1\}|\ge r$, so $l\ge 1$ and $S_{t+1}$ has at least two vertices and is indeed a star. Also, if $V(S_{t+1})\neq \{c_{t+1}\}\cup N$, let $V(S_{t+1})=\{c_{t+1},v_1,\dots,v_{l'}\}$. Then the degree of $v_{l'+1}$ in $G(t)$ is at most $\delta$, and removing it from $G(t+1)$ reduces the number of edges in $G(t+1)$ to less than $r$, so we have $|E_G(V(t+1))| < r+\delta < r+\delta + n$. Hence we terminate the algorithm on the next run through, and $S_{t+1}$ is the last star chosen.

Now we check the first two conditions of the lemma. In choosing the star $S_{t+1}$ from $G(t)$, we ensure that $G(t+1)$ has at least $r$ edges, so Condition \ref{Con1_V0_r_edges} is satisfied. To show Condition \ref{Con1_star_edges} is satisfied, let the set of stars produced by the algorithm be $\{S_1, \dots, S_k\}$. Since all the vertices of $N_{G(t-1)}(c_t)$ have degree at most $\delta$ in $G$ and hence in $G(t-1)$, for $1\le t \le k$ we have
\begin{align*}
|E_G(V(t-1))|-|E_G(V(t))|&\le |N_{G(t-1)}(c_t)|+(\delta-1)(|V(S_t)|-1)\\
&< n + \delta |E(S_t)|.
\end{align*}
If $t\neq k$, we also have $V(S_t) = \{c_t\}\cup N_{G(t-1)}(c_t)$, and so instead we get
\[
|E_G(V(t-1))|-|E_G(V(t))|\le \delta |E(S_t)|.
\]
Hence we have
\begin{align*}
 r + n + \delta &\ge |E_G(V(k))|\\
&= |E_G(V(0))| - \sum_{i=1}^k \left(|E_G(V(t-1))|-|E_G(V(t))|\right)\\
&\ge |E|-\delta |D_s| - \sum_{i=1}^k \delta|E(S_i)| - n\\
&\ge n\delta/2 - \delta n z(5,\delta) - \delta\left| E(F_S)\right| - n.
\end{align*}
Rearranging this, we obtain
\[
 \left| E(F_S)\right| \ge n(1/2-z(5,\delta)-2/\delta) - 1 - r/\delta.
\]
This is the statement of the Condition \ref{Con1_star_edges} of the lemma. For the final condition, since $D \subseteq D_s \cup V_b \subseteq (V\setminus V(F_S))\cup V_1$, every vertex of $G$ has at least $5$ edges to $(V\setminus V(F_S))\cup V_1$, as required. In the case of a vertex in $V_1$, these must all be to $V\setminus V(F_S)$ -- indeed, $V_1\subseteq V_b$, and so $V_1$ has no internal edges.
\end{proof}

In the next lemma, we use the structure given by Lemma \ref{lem_stars} to find a more precise structure in a graph $G$ of high minimal degree. Recall that in Subsection \ref{subsec_graph_lemmas} $m(n,r_1,r_2)$ was defined as the least integer $r$ such that if a graph $G=(V,E)$ on $n$ vertices has at least $r$ edges, $V$ can be partitioned into subsets $V_1$ and $V_2$ such that $|E_G(V_1)|\ge r_1$ and $|E_G(V_2)|\ge r_2$. Note that if $n' \le n$ we have $m(n',r_1,r_2)\le m(n,r_1,r_2)$; indeed, if there exists a graph on $n'$ vertices with $r$ edges which shows that $m(n',r_1,r_2)>r$, the same graph together with $n-n'$ isolated vertices shows that $m(n,r_1,r_2)>r$.
\begin{lemma}\label{lem_partition}
Let $\delta$, $n$, $r$, $r_1$ and $r_2$ be positive integers such that $m(n,r_1,r_2)+n \le r\le \delta n/2$, and $\delta \ge 5$. Let $G=(V,E)$ be a graph on $n$ vertices with minimum degree at least $\delta$. Then there exists a star forest $F_S$, a vertex set $V_1$ consisting of some of the centres of stars in $F_S$, and a forest $F$ such that:
\begin{enumerate}
\item \label{Con_Struct} $F_S\subseteq F \subseteq G$, and $F$ is a spanning forest for $G$.
\item \label{Con_V0_total_edges}$|E_G(V\setminus V(F_S))|\ge r$,
\item \label{Con_star_edges} $\left| E(F_S) \right| \ge n(1/2-z(5,\delta)-2/\delta) - 1 - r/\delta$,
\item \label{Con_F_components} every component of $F$ which is not contained in $V(F_S)$ has size at least $3$,
\item \label{Con_only_star} for all vertices $v\in V(F_S)\setminus V_1$, $E_{F_S}(\{v\})=E_F(\{v\})$,
\item \label{Con_2_nonC}for all vertices $v\in V\setminus V_1$, there are at least $2$ edges from $v$ to $(V\setminus V(F_S)) \cup V_1$ which are not in $F$,
\item \label{Con_2_C}for all vertices $v\in V_1$, there are at least $2$ edges from $v$ to $V\setminus V(F_S)$ which are not in $F$,
\item \label{Con_V0_part}$V\setminus V(F_S)$ has a partition into sets $U_1$ and $U_2$ such the $E_G(U_1)\setminus E(F)\ge r_1$, and $E_G(U_2)\setminus E(F)\ge r_2$.
\end{enumerate}
\end{lemma}
\begin{proof}
First, applying Lemma \ref{lem_stars} to the graph $G$ gives us a star forest $F_S$ and a set $V_1$ consisting of some of the centres of $F_S$ such that the conclusion of Lemma \ref{lem_stars} applies. Next, we shall select the forest $F$. We define vertex sets $V_0$ and $V_2$ by $V_0 = V\setminus V(F_S)$ and $V_2=V(F_S)\setminus V_1$; so $(V_0, V_1, V_2)$ is a partition of $V$.

Let $G_1$ be the graph on vertex set $V_0\cup V_1$, and edge set $E_G(V_0)\cup E_G(V_0, V_1)$. Note that by Condition \ref{Con1_centre_set} of Lemma \ref{lem_stars}, $G_1$ has minimum degree at least $5$. So applying Lemma \ref{lem_all_colours} to the graph $G_1$ gives us a colouring $f:E_G(V_0)\cup E_G(V_0,V_1)\mapsto \{1,2\}$ with each colour appearing twice at every vertex. Now, let $G_2$ be the graph on vertex set $V_0\cup V_1$ with edge set $\{e\in E_G(V_0)\cup E_G(V_0,V_1): f(e)=2\}$. We define $F'$ to be any spanning forest for the graph $G_2$ such that the components of $F'$ are the same as the components of $G_2$, and $F$ to be the forest with vertex set $V$ and edge set $E(F')\cup E(F_S)$.

We claim that all the conditions of the lemma except the last now hold. For Condition \ref{Con_Struct}, we need to check that $F$ is a spanning forest of $G$. $F$ is the union of a spanning forest for $V_0\cup V_1$, together with a set of stars which span $V_2$ and each have at most one vertex in $V_1$. Hence $F$ is a spanning forest for $V$.

Conditions \ref{Con_V0_total_edges} and \ref{Con_star_edges} follow immediately from the corresponding conditions in Lemma \ref{lem_stars}. For Condition \ref{Con_F_components}, all vertices in $V_0$ have at least two edges $e$ in $E_G(V_0)\cup E_G(V_0,V_1)$ with $f(e)=2$, and so are in components of $G_2$, and hence of $F$, of size at least $3$. Condition \ref{Con_only_star} follows since $E(F)= E(F')\cup E(F_S)$, and $F'$ does not have any edges incident with $V_2$.

For vertices in $V_2$, Condition \ref{Con_2_nonC} is then immediate from Condition \ref{Con1_centre_set} in Lemma \ref{lem_stars}; in fact, all vertices $v\in V_2$ have at least five edges to $V_0\cup V_1$, and at most one of these edges is in $F$. Condition \ref{Con_2_nonC} also holds for vertices in $V_0$, since every vertex $v\in V_0$ has $2$ edges $e$ to $V_0\cup V_1$ with $f(e)=1$, and these edges cannot be in $F$. Similarly, Condition \ref{Con_2_C} holds because every vertex $v\in V_1$ has $2$ edges $e$ to $V_0$ with $f(e)=1$.

It remains only to choose the partition $(U_1, U_2)$ of $V_0$ so as to satisfy Condition \ref{Con_V0_part}. Let $G'$ be the graph on vertex set $V_0$, with edge set $E_G(V_0)\setminus F$. Now, $E_G(V_0)$ contains at least $r\ge m(n,r_1,r_2)+n$ edges, and $F$ has fewer than $n$ edges overall, so $|E(G')|=|E_G(V_0)\setminus F|\ge m(n,r_1,r_2)$. Since $G'$ has at most $n$ vertices, this guarantees a partion of $V(G')$ into sets $U_1$ and $U_2$ so that $|E_{G'}(U_1)|\ge r_1$ and $E_{G'}(U_2)\ge r_2$. This partition of $V_0$ satisfies the final condition of the lemma, completing the proof.
\end{proof}

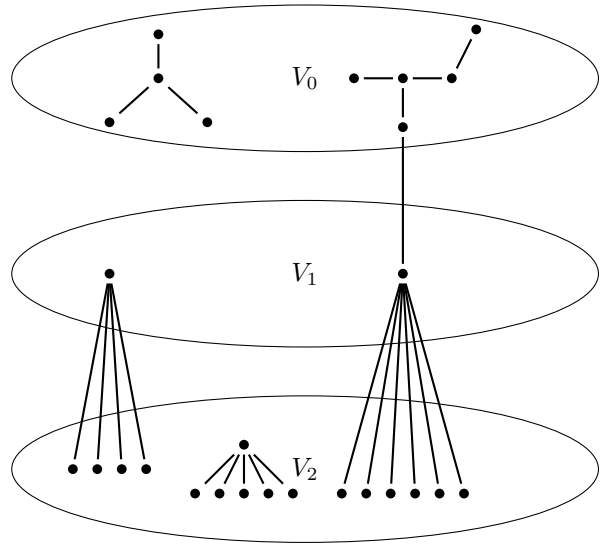
\begin{figure}[ht]
\centering{
\begin{tikzpicture}[scale=0.65,thick]
  \begin{pgfonlayer}{nodelayer}
   \node (a) at (4,9){\(V_0\)};
   \node (b) at (4,5){\(V_1\)};
   \node (c) at (4,1){\(V_2\)};
   \node (a_0) at (0,8.1){};
   \node (a_1) at (1,9){};
   \node (a_2) at (2,8.1){};
   \node (a_3) at (1,9.9){};
   \node (c_0) at (0,5){};
   \node (c_1) at (2.75,1.5){};
   \node (c_2) at (6,5){};
   \foreach \x in {0,...,3}{\node (l0_\x) at ($(-0.75,1)+(0.5*\x,0)$){};}
   \foreach \x in {0,...,4}{\node (l1_\x) at ($(1.75,0.5)+(0.5*\x,0)$){};}
   \foreach \x in {0,...,5}{\node (l2_\x) at ($(4.75,0.5)+(0.5*\x,0)$){};}
\node (d_0) at (6,8){};
\node (d_1) at (6,9){};
\node (d_2) at (5,9){};
\node (d_3) at (7,9){};
\node (d_4) at (7.5,10){};

   \begin{scope}
   \draw (a) ellipse (6cm and 1.5cm);
   \draw (b) ellipse (6cm and 1.5cm);
   \draw (c) ellipse (6cm and 1.5cm);
   \end{scope}
\end{pgfonlayer}

\foreach \x in {0,...,2}{\fill (c_\x) circle (0.1);}
\foreach \x in {0,...,3}{\fill (a_\x) circle (0.1);}
\foreach \x in {0,...,3}{\fill (l0_\x) circle (0.1);}
\foreach \x in {0,...,4}{\fill (l1_\x) circle (0.1);}
\foreach \x in {0,...,5}{\fill (l2_\x) circle (0.1);}
\foreach \x in {0,...,4}{\fill (d_\x) circle (0.1);}

\draw (a_0) to (a_1);
\draw (a_1) to (a_2);
\draw (a_1) to (a_3);
\foreach \x in {0,...,3}{\draw (l0_\x) to (c_0);}
\foreach \x in {0,...,4}{\draw (l1_\x) to (c_1);}
\foreach \x in {0,...,5}{\draw (l2_\x) to (c_2);}
\draw (c_2) to (d_0);
\draw (d_0) to (d_1);
\draw (d_1) to (d_2);
\draw (d_1) to (d_3);
\draw (d_3) to (d_4);

\end{tikzpicture}}
\caption{The structure of the forest F}
\label{fig_F_struct}
\end{figure}

\section{Edge colouring a graph with large minimum degree}\label{sec_form_to_anti}
The aim of this section is to demonstrate an algorithm which, given a graph $G$ with large minimum degree, uses the partition guaranteed by Lemma \ref{lem_partition} to label $G$ as Lemma \ref{lem_min_suffices} demands. In fact, we shall prove that Lemma \ref{lem_min_suffices} holds for $k_1$ and $k_2$ coprime odd integers, both at least $9$, with the constant $c=2k_1k_2+k_2$.

For the rest of this section, we fix:
\begin{itemize}
\item coprime odd integers $k_1$ and $k_2$, both at least $9$,
\item an integer $\delta$,
\item a graph $G=(V,E)$ on $n$ vertices with minimum degree at least $\delta$,
\item $r_1 = (k_1k_2+1)n$, $r_2 = (k_1+1)n$, and $r = \max(2(k_1k_2+k_1)n, m(n, r_1, r_1)+n)$,
\item a choice of $F_S$, $F$, $V_1$, $U_1$ and $U_2$ to satisfy the conclusions of Lemma \ref{lem_partition} for $\delta$, $n$, $r$, $r_1$, $r_2$ and $G$.
\item a label set $L$ of size $|E|$, containing $[1,(2k_1k_2+k_2)n]$,
\item and a function $g:V\mapsto \mathbb N$.
\end{itemize}
To prove Lemma \ref{lem_min_suffices} for $k_1$ and $k_2$ with the constant $c=2k_1k_2+k_2$, our task is to show that $G$ has a bijective $g$-antimagic colouring $f:E\mapsto L$ with no sum $s_{(f,g)}(v)$ divisible by $k_1k_2$, provided $\delta$ is larger than some constant. Before we begin to label $E$, as before we define $V_0 = V\setminus V(F_S)$, and $V_2 = V(F_S)\setminus V_1$. Also, we define $U_3$ to be $\{v \in V_0: E_G(\{v\}, V_0)\}\subseteq F\}$, and we write the stars of $F_S$ as $S_1, \dots, S_k$, with centres $c_1,\dots,c_k$. Further, we define a new partition on the edge set of $G$ as follows:
\begin{align*}
E_1&=E_G(V_2,V)\setminus F,\\
E_2&=E_G(V_1,V_0\cup V_1)\setminus F,\\
E_3&=E_G(V_0)\setminus F,\\
E_4&=F \setminus E(F_S),\\
E_5&= E(F_S).
\intertext{Note that the $E_i$ are disjoint, and their union is $E$. Now we define a partition of the labels of $L$ as follows:}
L_1&=[1,k_1k_2n] \setminus \{k_1k_2i : 1 \le i \le |F|\},\\
L_2&=[\max(L_1)+1,\max(L_1)+k_1k_2(2|V_1|+|U_3|)],\\
L_3&=[\max(L_2)+1, \max(L_2)+(|V_0\setminus U_3|-1)(k_1k_2+k_2)],\\
L_4&=\{k_1k_2i : 1 \le i \le |F|\},\\
L_T&= L \setminus (L_1\cup L_2 \cup L_3 \cup L_4).
\end{align*}
Note that the $L_i$ are all disjoint. Also, note that at most half of the vertices in $V_1\cup V_2$ are centres of stars in $F_S$, since every star has at least one vertex which is not the centre. Since $V_1$ is a subset of the centres of stars in $F_S$, we have $|V_1|\le |V_2|$, and so
\begin{align*}
 \max(L_1\cup L_2\cup L_3 \cup L_4)&= k_1k_2n + 2k_1k_2|V_1| + k_1k_2 |U_3|+\\
&+ (k_1k_2+k_2)(|V_0\setminus U_3|-1)\\
&< k_1k_2n + k_1k_2|V_2\cup V_1| + (k_1k_2+k_2)|V_0|\\
&\le (2k_1k_2+k_2)n.
\end{align*}
Since $[1,cn]\subseteq L$, we have $L_1\cup L_2 \cup L_3 \cup L_4\subseteq L$. Also, by Condition \ref{Con_V0_total_edges} of Lemma \ref{lem_partition} we have $|E_3| \ge (2k_1k_2+k_2)n-|F|\ge |L_1\cup L_2 \cup L_3|$. Hence we have
\begin{align*}
 |L_T| &= |L|- |L_1\cup L_2\cup L_3| - |L_4|\\
&= |E| - |L_1\cup L_2\cup L_3| - |E_4| - |E_5|\\
&= |E_1\cup E_2 \cup E_3| - |L_1\cup L_2\cup L_3|\\
&\ge |E_1\cup E_2|.
\end{align*}
The $L_i$ for $1 \le i \le 4$ are chosen carefully, and when we need some control over the label we use for an edge it will be these sets that are useful. The set $L_T$ we have no control over whatsoever; we shall use labels from $L_T$ when it is unimportant what label an edge receives.

\begin{table}[ht]
	\centering
		\begin{tabular}{c|c|l|l}
			Lemma & \pbox{100cm}{Edge\\set} & Label set & Aim \\
			\hline
			\ref{lem_E1_label} & $E_1$ & \pbox{100cm}{Some of $L_1$,\\ some of $L_T$} & \pbox{100cm}{Centres of stars in $V_2$ have sum equal\\to $1\pmod{k_1k_2n}$, other vertices\\in $V_2$ have sum ${m\pmod{k_1k_2n}}$}\\
			\hline
			\ref{lem_E2_label} & $E_2$ & \pbox{100cm}{Some of $L_2$,\\ some of $L_T$} & \pbox{100cm}{Vertices in $V_1$ have sum equal to \\$1\pmod{k_1k_2}$.\\ Vertices in $U_3$ have sum not equal\\to $0$ or ${1\pmod{k_1}}$, and not too\\many in any class modulo $k_1$.}\\
			\hline
			\ref{lem_E3_label} & $E_3$ & \pbox{100cm}{All of $L_3$,\\ rest of $L_1$, $L_2$, $L_T$} & \pbox{100cm}{Vertices in $U_1\setminus U_3$ have sum not\\equal to $0$ or ${1\pmod{k_1}}$, and not\\too many in any class modulo $k_1$,\\and similarly for $U_2\setminus U_3$ and $k_2$.}\\
			\hline
			\ref{lem_E4_label} & $E_4$ & Some of $L_4$ & Antimagic on $V_0$\\
			\hline
			\ref{lem_E5_label} & $E_5$ & Rest of $L_4$ & Antimagic on centres of stars\\
		\end{tabular}
	\caption{Strategy for labelling $G$}
	\label{tab:Strat}
\end{table}

Now we are ready to begin labelling $E$. We shall do this by labelling the edge sets $E_i$ in turn. We do this in a series of lemmas; each lemma takes the labelling guaranteed by the last, and extends it to label another set of edges. Table \ref{tab:Strat} summarises the edges labelled in each lemma, the labels used, and the aim for the partials sums of the labelling.

To give an overview of the labelling, we shall first go up the graph (as depicted in Figure \ref{fig_F_struct}), labelling all edges not in the spanning forest $F$. Here, we shall control the partial sums at vertices modulo $k_1$ and $k_2$; as we go up the graph, we shall be able to achieve less and less precise control over these sums. Then, we shall come back down the graph labelling $F$, essentially labelling greedily to avoid sums being equal. As we get nearer the end, we have fewer labels to choose from, but our more precise control over the partial sums so far and the special structure of $F_S$ compensates for this lack of choice.

First we shall label the set $E_1$, with labels from $L_1$ and $L_T$. Our aim here is that the vertices of $V_2$ receive specified partial sums modulo $k_1k_2n$; the sums at centres of stars will be congruent to $1$, while the sums at other vertices will be congruent to $0$ modulo $k_1$ and $1$ modulo $k_2$. We can achieve this without much difficulty, using two edges from a vertex $v\in V_2$ to $V_0\cup V_1$ to fix the sum at $v$.

We define the integer $m$ to be the unique element of $\{0,\dots,k_1k_2-1\}$ with $m\equiv 0 \pmod{k_1}$ and $m\equiv 1 \pmod{k_2}$.
\begin{lemma}\label{lem_E1_label}
There is an injective labelling $f_1:E_1\mapsto L_1\cup L_T$ such that
\begin{enumerate}
\item if $v$ is a vertex of $V_2$ which is the centre $c_i$ of one of the stars $S_i$, then $s_{(f_1,g)}(v)\equiv 1 \pmod{k_1k_2n}$,
\item if $v$ is a vertex of $V_2$ which is not the centre of one of the stars $S_i$, then $s_{(f_1,g)}(v)\equiv m \pmod{k_1k_2n}$.
\end{enumerate}
\end{lemma}
\begin{proof}
 By Condition \ref{Con_2_nonC} of Lemma \ref{lem_partition}, we can choose a set $E'_1\subseteq E_1$ consisting of two edges from each $v\in V_2$ to $V_0\cup V_1$. We label all edges in $E_1\setminus E_1'$ with distinct and otherwise arbitrary labels from $L_T$. There are enough labels in $L_T$ to do this, since $|L_T| > |E_1|$. Let $f''_1:E_1\setminus E'_1 \mapsto L_T$ be the labelling this gives us. Now we step through the vertices of $V_2$ in any order, labelling the edges of $E'_1$ adjacent to each vertex as we come to it. When we come to a vertex $v$, suppose the edges at $v$ in $E'_1$ are $e_1$ and $e_2$. We label $e_1$ and $e_2$ with labels $l_1$ and $l_2$, obeying the following conditions:
\begin{enumerate}
\item $l_1$ and $l_2$ are not the same as any label already used in the labelling of $E'_1$,
\item $l_1$ and $l_2$ are labels of $L_1$,
\item 
$s_{(f''_1,g)}(v)+l_1+l_2\equiv 
\begin{cases}
1 \pmod{k_1k_2n}: \textrm{ if $v$ is $c_i$ for some $1\le i \le k$}\\
m \pmod{k_1k_2n}: \textrm{ otherwise}.
\end{cases}
$
\end{enumerate}
We claim this is always possible. Indeed, when we reach a vertex $v \in V_2$, at most $2|V_2|-2$ labels from $L_1$ have been used, so at least 
\[
 k_1k_2n-|F|-2|V_2|+2 > (k_1k_2-3)n > k_1k_2n/2
\]
labels remain, each of which is in a different congruency class modulo $k_1k_2n$. Hence there are two unused labels $l_1$ and $l_2$ in $L_1$ satisfying the conditions.
Let $f'_1:E'_1\mapsto L_1$ be the labelling this process gives. We define $f_1$ to be $f'_1$ on $E'_1$, and $f''_1$ on $E_1\setminus E'_1$. Then $f_1$ is a labelling of $E_1$ satisfying the conditions of the lemma.
\end{proof}

Next we take the labelling given by Lemma \ref{lem_E1_label}, and extend it to also label $E_2$. $E_2$ will be labelled using labels from $L_2$ and $L_T$. Our aim here is that vertices in $V_1$ recieve partial sums congruent to $1$ modulo $k_1k_2$, while the partial sums at vertices of $U_3$ are not congruent to $0$ or $1$ modulo $k_1$, and there are not too many in any other congruency class modulo $k_1$. This is achieved using Lemma \ref{lem_AB_colour}, which precisely guarantees us a labelling of this form.
\begin{lemma}\label{lem_E2_label}
 There is an injective labelling $f_2:E_1\cup E_2 \mapsto L_1\cup L_2\cup L_T$ such that
\begin{enumerate}
\item if $v$ is a vertex of $V_2$ which is the centre $c_i$ of one of the stars $S_i$, then $s_{(f_2,g)}(v)\equiv 1 \pmod{k_1k_2n}$,
\item if $v$ is a vertex of $V_2$ which is not the centre of one of the stars $S_i$, then $s_{(f_2,g)}(v)\equiv m \pmod{k_1k_2n}$,
\item if $v$ is a vertex of $V_1$, then $s_{(f_2,g)}(v)\equiv 1 \pmod{k_1k_2}$,
\item the induced colouring of the vertices of $U_3$ satisfies the following conditions:
\begin{enumerate}[a)]
 \item $n_{(f_2,g,U_3,k_1)}(0)=n_{(f_2,g,U_3,k_1)}(1)=0$,
 \item for each $2\le i\le k_1-1$, $n_{(f_2,g,U_3,k_1)}(i)\le |U_3|/(k_1-4)+2k_1-3$.
\end{enumerate}
\end{enumerate}
\end{lemma}
\begin{proof}
First, applying Lemma \ref{lem_E1_label} gives us an injective labelling $f_1:E_1\mapsto L_1\cup L_T$ satisfying the conclusions of that lemma. Let $L'_T$ be the set of labels in $L_T$ which are not in the image of $f_1$. Now, for a vertex $v\in V$, let $g'(v) = s_{(f_1,g)}(v)$. Define a graph $G'$ with vertex set $V_0\cup V_1$ and edge set $E_2$. Now, we apply Lemma \ref{lem_AB_colour}. In the statement of that lemma, we have a graph $G$, vertex sets $A$, $B$ and $B'$, integers $k_1$ and $k_2$, and a label set $L$; here we use the graph $G'$, the sets $V_1$, $V_0$ and $U_3$, the integers $k_1$ and $k_2$, and the label set $L'=L_2\cup L'_T$. To check the conditions of Lemma \ref{lem_AB_colour} hold, by Condition \ref{Con_2_C} of Lemma \ref{lem_partition} every vertex in $V_1$ has at least $2$ edges to $V_0$. Since vertices in $U_3$ have no edges to $V_0$, by Condition \ref{Con_2_nonC} of Lemma \ref{lem_partition} every vertex in $U_3$ has at least one edge to $V_1$. Since $|L'_T|\ge |L_T|-|E_1|\ge |E_2|$, $L'$ contains at least $|L_2|+|L'_T|\ge |E(G')|+4k_1k_2|A|+k_1|B'|$ labels. Finally, $L_2$ contains the required number of labels in each congruency class modulo $k_1$ and $k_1k_2$, and so Lemma \ref{lem_AB_colour} does indeed apply. We set the function $g:V(G')\mapsto \mathbb N$ in that lemma to be $g'$, and we set the function $t:V(G')\mapsto \mathbb N$ to be constantly $1$. Then by Lemma \ref{lem_AB_colour} there is an injective labelling $f'_2:E_2\mapsto L'$ such that
\begin{enumerate}
\item for each $v$ in $V_1$, $s_{(f'_2,g')}(v)\equiv 1 \pmod{k_1k_2}$,
\item $n_{(f'_2,g',U_3,k_1)}(0)=n_{(f'_2,g',U_3,k_1)}(1)=0$,
\item for each $2\le i\le k_1-1$, $n_{(f'_2,g',U_3,k_1)}(i)\le |U_3|/(k_1-4)+2k_1-3$.
\end{enumerate}
Now, we define the labelling $f_2:E_1\cup E_2 \mapsto L_1\cup L_2\cup L_T$ by setting $f_2=f_1$ on $E_1$ and $f_2=f'_2$ on $E_2$. Since $f'_2$ does not label any edge incident with $E_1$, the properties first two conditions of the lemma follow from the corresponding conditions for $f_1$. Also, since for all vertices $v$ we have $s_{(f_2,g)}(v)=s_{(f'_2,g')}(v)$, the other conditions in the lemma follow from the above conditions on $f'_2$.
\end{proof}

In the next lemma, we take the labelling given by Lemma \ref{lem_E2_label}, and extend it to also label $E_3$. This will be done with the remainder of the sets $L_1$, $L_2$ and $L_T$, and the whole of the label set $L_3$. We define $U'_1 = U_1\setminus U_3$ and $U'_2 = U_2\setminus U_3$. The aim is that the partial sums at vertices in $U'_1$ are not congruent to $0$ or $1$ modulo $k_1$, and there are not too many in any other congruency class modulo $k_1$, and similarly for $U'_2$ and $k_2$. To achieve this we shall use Lemma \ref{lem_k12colour}, which guarantees us a labelling to achieve precisely these conditions.

\begin{lemma}\label{lem_E3_label}
 There is a bijective labelling $f_3:E_1\cup E_2\cup E_3 \mapsto L_1\cup L_2\cup L_3\cup L_T$ such that
\begin{enumerate}
\item if $v$ is a vertex of $V_2$ which is the centre of a star in $F_S$, then $s_{(f_3,g)}(v)\equiv 1 \pmod{k_1k_2n}$,
\item if $v$ is a vertex of $V_2$ which is the centre of a star in $F_S$, then $s_{(f_3,g)}(v)\equiv m \pmod{k_1k_2n}$,
\item if $v$ is a vertex of $V_1$, then $s_{(f_3,g)}(v)\equiv 1 \pmod{k_1k_2}$,
\item the induced colouring of the vertices of $U_3$ satisfies the following conditions:
\begin{enumerate}[a)]
 \item $n_{(f_3,g,U_3,k_1)}(0)=n_{(f_3,g,U_3,k_1)}(1)=0$,
 \item for each $2\le i\le k_1-1$, $n_{(f_3,g,U_3,k_1)}(i)\le |U_3|/(k_1-4)+2k_1-3$,
\end{enumerate}
\item for $i=1$ and $2$, the induced colouring of the vertices of $U'_i$ satisfies the following conditions:
\begin{enumerate}[a)]
\item $n_{(f_3,g,U'_i,k_i)}(0)=n_{(f_3,g,U'_i,k_i)}(1)=0$,
\item for each $2\le j\le k_i-1$, $n_{(f_3,g,U'_i,k_i)}(j)\le |U'_i|/(k_i-3) + k_i + 2$.
\end{enumerate}
\end{enumerate}
\end{lemma}
\begin{proof}
First, applying Lemma \ref{lem_E2_label} gives us an injective labelling $f_2:E_1\cup E_2 \mapsto L_1\cup L_2\cup L_T$ satisfying the conclusions of that lemma. Let $L'$ be those labels in $L_1\cup L_2 \cup L_3 \cup L_T$ which are not in the image of $f_2$. Note that since $|E_4|+|E_5|=|L_4|$, we have $|L'|=|E_3|$, and also note $L_3\subseteq L'$. Let $g':V\mapsto \mathbb N$ be given by $g'(v)=s_{(f_2,g)}(v)$. We define a graph $G'$, having vertex set $U'_1\cup U'_2$, and edge set $E_3$. We wish to label $E_3$ with $L'$ using Lemma \ref{lem_k12colour}. In the statement of that lemma, we have a graph $G$, odd integers $k_1$ and $k_2$, a label set $L$, and vertex sets $V_1$ and $V_2$; here we use the graph $G'$, the integers $k_1$ and $k_2$, the label set $L'$ and the vertex sets $U'_1$ and $U'_2$. To show that Lemma \ref{lem_k12colour} does indeed apply, note that since $U'_1$ and $U'_2$ are disjoint from $U_3$, there are no isolated vertices in $G'$. Also, Condition \ref{Con_V0_part} of Lemma \ref{lem_partition} guarantees that $E_{G'}(U'_1)\ge (k_1k_2+1)n \ge (k_1k_2+1)|V(G')|$, and $|E_{G'}(U'_2)|\ge (k_1+1)n\ge (k_1+1)|V(G')|$. $L'$ contains at least as many labels in each congruency class modulo $k_1k_2$ and $k_1$ as Lemma \ref{lem_k12colour} requires, since $L_3$ does and $L_3\subseteq L'$. So Lemma \ref{lem_k12colour} applies. We set the function $g:V\mapsto \mathbb N$ in that lemma to be $g'$. Then by Lemma \ref{lem_k12colour} there is a bijective labelling $f'_3: E_3 \mapsto L'$ such that for $i =1$ and $2$ we have:
\begin{enumerate}
\item $n_{(f'_3,g',U'_i,k_i)}(0)=n_{(f'_3,g',U'_i,k_i)}(1)=0$,
\item for each $2\le j\le k_i-1$, $n_{(f'_3,g',U'_i,k_i)}(j)\le |U'_i|/(k_i-3) + k_i + 2$.
\end{enumerate}
Now, we define the labelling $f_3:E_1\cup E_2\cup E_3 \mapsto L_1\cup L_2\cup L_3\cup L_T$ to be equal to $f_2$ on $E_1\cup E_2$, and equal to $f'_3$ on $E_3$. Since $f'_3$ labels no edge incident with $V_1$, $V_2$ or $U_3$, the properties we need for $f_3$ on those sets are inherited from the corresponding properties of $f_2$. Also, for each $v \in U'_1 \cup U'_2$ we have $s_{(f'_3,g')}(v)=s_{(f_2,g)}(v)$. Hence the conditions we need on the sums in $U'_1$ and $U'_2$ follow from the above conditions on $f'_3$.
\end{proof}

At this stage, only the forest $F$ remains unlabelled, and the labels $E_4=\{k_1k_2,\dots,k_1k_2|F|\}$ remain to label $F$. In the next lemma, we extend the labelling from Lemma \ref{lem_E3_label} to label $E_4$ as well. This shall be done using some of the labels from $L_4$. The aim is that the vertices of $V_0$ receive distinct overall sums; to achieve this we shall use a greedy algorithm. This works because we have ensured that there are not too many vertices of $V_0$ with partial sums in any congruency class modulo $k_1k_2$, and all the labels in $E_4$ are divisible by $k_1k_2$, so each vertex in $V_0$ has a potential conflict with only fairly few other vertices in $V_0$. It is at this stage that we shall need $\delta$ to be large, to guarantee $E_5$ is large and so that there are always enough labels remaining in $L_4$ to pick an appropriate one to label an edge.
\begin{lemma}\label{lem_E4_label}
Suppose the following equation holds for $\delta$:
\begin{align}\label{eq_for_delta}
 &\delta\left(1/2-z(5,\delta)-\frac{2}{\min(k_1-4,k_2-3)}\right) \ge \\
\notag &\ge \max(2k_1k_2+k_2,m'(k_1k_2+1,k_2+1)+1)+6k_1+2k_2+2.
\end{align}
Then there is an injective labelling $f_4:E_1\cup E_2\cup E_3\cup E_4 \mapsto L$ such that:
\begin{enumerate}
\item the image of $f_4$ includes $L_1\cup L_2 \cup L_3 \cup L_T$,
\item if $v$ is a vertex of $V_2$ which is the centre $c_i$ of one of the stars $S_i$, then $s_{(f_4,g)}\equiv 1 \pmod{k_1k_2n}$,
\item if $v$ is a vertex of $V_2$ which is not the centre of one of the stars $S_i$, then $s_{(f_4,g)}\equiv m \pmod{k_1k_2n}$,
\item if $v$ is a vertex of $V_1$, then $s_{(f_4,g)}(v)\equiv 1 \pmod{k_1k_2}$,
\item if $v$ is a vertex of $V_0$, $s_{(f_4,g)}(v)$ is not congruent to $0$, $1$, or $m$ modulo $k_1k_2$,
\item for distinct vertices $v_1$ and $v_2\in V_0$, $s_{(f_4,g)}(v_1)\neq s_{(f_4,g)}(v_2)$.
\end{enumerate}
\end{lemma}
\begin{proof}
First, applying Lemma \ref{lem_E3_label} gives us a bijective labelling $f_3:E_1\cup E_2 \cup E_3 \mapsto L_1\cup L_2\cup L_3\cup L_T$ satisfying the conclusions of that lemma. Now we label $E_4$, using some of the labels from $L_4$. We do this by stepping through the edges of $E_4$ in any order, labelling each as we reach it. Let $E_4= \{e_1,\dots,e_r\}$. We define labellings $f^i$ for $0\le i \le r$, with $f^i$ being a labelling $f^i : E_1\cup E_2 \cup E_3\cup \{e_1,\dots, e_i\}\mapsto L$. For $i=0$, we define $f^0$ to be equal to $f_3$. For $i\ge 1$, define $f^i$ by setting $f^i=f^{i-1}$ on $E_1\cup E_2 \cup E_3 \cup \{e_1,\dots,e_{i-1}\}$, and letting $f^i(e_i)=l$ for any label $l$ satisfying the following conditions:
\begin{enumerate}
\item $l$ is in $L_4$, and not in the image of $f^{i-1}$,
\item \label{Con_no_clash_F} if $v$, $v'\in V_0$ with $v\in e$ and $v' \notin e$, $s_{(f^{i-1},g)}(v)+l \neq s_{(f^{i-1},g)}(v')$.
\end{enumerate}
We claim that such a label always exists. When we reach the edge $e=v_1v_2$, there are at least $|E_5|+1$ edges unlabelled, and correspondingly there are at least $|E_5|+1$ labels which obey the first condition. Of these, Condition \ref{Con_no_clash_F} rules out at most one label for each $v$, $v'\in V_0$ with $v\in e$, $v'\notin e$ and $s_{(f^{i-1},g)}(v)-s_{(f^{i-1},g)}(v')\in L_4$. An upper bound for the number of such vertices is the number of vertices $v'$ in $V_0\setminus e$ with $s_{(f_3,g)}(v')$ equal to $s_{(f_3,g)}(v_1)$ or $s_{(f_3,g)}(v_2)$ modulo $k_1k_2$, since all labels in $L_4$ are divisible by $k_1k_2$. From the conclusion of Lemma \ref{lem_E3_label}, the number of such vertices $v'$ is at most
\[
 2\left(\frac{|U_3|}{k_1-4} + 2k_1 - 3 + \frac{|U'_1|}{k_1-3} + k_1 + 2 + \frac{|U'_2|}{k_2-3} + k_2 + 2\right)-2.
\]
Indeed, the first term in the bracket represents the largest possible number of vertices $v\in U_3$ with $s_{(f_3,g)}(v) \equiv s_{(f_3,g)}(v_1) \pmod {k_1}$, the second the largest possible number of vertices $v\in U'_1$ with $s_{(f_3,g)}(v) \equiv s_{(f_3,g)}(v_1) \pmod {k_1}$, and the third the largest possible number of vertices $v\in U'_2$ with $s_{(f_3,g)}(v) \equiv s_{(f_3,g)}(v_1) \pmod {k_2}$. We may subtract $2$ because we need not consider the vertices $v_1$ and $v_2$. So there is a label that obeys the conditions so long as
\begin{equation}\label{eq_E5_large}
 |E_5|+1 \ge 2\left(\frac{|U_3|}{k_1-4} + \frac{|U'_1|}{k_1-3} + \frac{|U'_2|}{k_2-3} \right)+ 6k_1 + 2k_2.
\end{equation}
Assume for now that this equation holds. We define the labelling on $f_4: E_1\cup E_2 \cup E_3\cup E_4\mapsto L$ to be equal to $f^r$. We claim that $f_4$ satisfies the conditions of the lemma.

The first condition is satisfied, since the image of $f_4$ includes the image of $f_3$, which is $L_1\cup L_2 \cup L_3 \cup L_T$. The second and third conditions are guaranteed by the equivalent conditons for $f_3$, since $E_4$ has no edge incident with $V_1$. The fourth and fifth conditions hold for $f_3$, and so also for $f_4$, as $E_4$ is entirely labelled with labels divisible by $k_1k_2$. For the final condition, we claim the restrictions on the labelling of $E_5$ guarantee that $s_{(f_4,g)}(v_1)\neq s_{(f_4,g)}(v_2)$ for distinct vertices $v_1$ and $v_2$ in $V_0$. Indeed, let $e_j$ be the last edge incident with precisely one of $v_1$ and $v_2$ to be labelled; such an edge exists, by Condition \ref{Con_F_components} of Lemma \ref{lem_partition}. The conditions on the label given to $e_j$ guarantee that $s_{(f^j,g)}(v_1)\neq s_{(f^j,g)}(v_2)$, and hence $s_{(f_4,g)}(v_1)\neq s_{(f_4,g)}(v_2)$.

\noindent
To prove the lemma, it remains to check \eqref{eq_E5_large}. From Condition \ref{Con_star_edges} of Lemma \ref{lem_partition}, we have
\begin{equation}\label{eq_1}
 |E_5| = \left|\bigcup_{i=1}^k E(S_i)\right| \ge n(1/2-z(5,\delta)-2/\delta) - 1 - r/\delta.
\end{equation}
So to check \eqref{eq_E5_large} holds, it suffices to show that
\begin{equation}\label{eq_2}
 n(1/2-z(5,\delta)-2/\delta) - r/\delta \ge 2\left(\frac{|U_3|}{k_1-4} + \frac{|U'_1|}{k_1-3} + \frac{|U'_2|}{k_2-3} \right)+ 6k_1 + 2k_2.
\end{equation}
Now, $|U_3| + |U'_1| + |U'_2| = |V_0| \le n$, and so to establish \eqref{eq_2} it suffices to show that
\begin{equation}\label{eq_3}
 n\left(1/2 - z(5,\delta) - \frac{2}{\min(k_1-4,k_2-3)} - 2/\delta\right) - r/\delta \ge 6k_1 + 2k_2.
\end{equation}
Rearranging this equation, and multiplying by $\delta/n$, \eqref{eq_3} is equivalent to
\begin{equation}\label{eq_4}
 \delta/2-\delta z(5,\delta) - \frac{2\delta}{\min(k_1-4,k_2-3)} \ge r/n + (6k_1 + 2k_2)\delta/n+2.
\end{equation}
However, $G$ is a graph on $n$ vertices with minimum degree at least $\delta$, so $\delta/n<1$, and so for \eqref{eq_4} to hold it is enough that
\begin{equation}\label{eq_5}
 \delta\left(1/2-z(5,\delta)-\frac{2}{\min(k_1-4,k_2-3)}\right) \ge r/n + 6k_1+2k_2+2.
\end{equation}
Now, $r = \max((2k_1k_2+k_2)n,m(n,r_1,r_2)+n)$; but from Corollary \ref{Cor_good_partition} we have $m(n,r_1,r_2)=m(n,(k_1k_2+1)n,(k_1+1)n)\le m'(k_1k_2+1,k_1+1)n$. So for \eqref{eq_5} to hold it is enough that
\begin{align*}
 &\delta\left(1/2-z(5,\delta)-\frac{2}{\min(k_1-4,k_2-3)}\right) \ge \\
\notag &\ge \max(2k_1k_2+k_2,m'(k_1k_2+1,k_2+1)+1)+6k_1+2k_2+2.
\end{align*}
This is precisely the assumption of the lemma, and the proof is complete.
\end{proof}
Lemma \ref{lem_E4_label} leaves only $E_5$ unlabelled, and the unused labels are a subset of $L_4$. We wish label $E_5$ so that the sums at the centres of the stars in $F_S$ are all distinct. To achieve this, we use the following simple lemma:
\begin{lemma}\label{lem_star_label}
Let $S_1,\dots,\,S_k$ be vertex disjoint stars with centres $c_1,\dots,\,c_k$, let $L$ be a set of integers of size $\left|\bigcup_{i=1}^k E(S_i)\right|$, and let $g$ be a function $g:\bigcup_{i=1}^k V(S_i) \mapsto \mathbb N$. Then there exists a bijective edge labelling $f: \bigcup_{i=1}^k E(S_i) \mapsto L$ such that the sums $s_{(f,g)}(c_i)$ are distinct.
\end{lemma}
\begin{proof}
We prove this by induction on $k$; for $k=1$ the result is trivial. For $k\ge 2$, let $L=\{l_1,\dots,\,l_r\}$ with $l_1<\dots<l_r$. For $1 \le i \le k$ let $n_i = g(c_i)+\sum_{i=1}^{|E(S_i)|} l_i$. Without loss of generality, $n_k$ is the smallest of the $n_i$. We label $E(S_k)$ with $\{l_1,\dots,l_{|E(S_k)|}\}$ in any order. By the induction hypothesis, there is a labelling of $\bigcup_{i=1}^{k-1} E(S_i)$ with the rest of $L$ so that the sums at $c_1,\dots,\,c_{k-1}$ are distinct. Also, for this labelling we have $s_{(f,g)}(c_i)>n_i\ge n_k=s_{(f,g)}(c_k)$ for $i\neq k$, and so in fact the sums at the centres of all the stars are distinct.
\end{proof}

Using the labelling guaranteed by Lemma \ref{lem_E4_label} and applying Lemma \ref{lem_star_label} to label $E_5$, the edge set of the stars in $F_S$, we show that we can find an antimagic colouring of $G$:
\begin{lemma}\label{lem_E5_label}
 Suppose $\delta$ satisfies \eqref{eq_for_delta}. Then there is a bijective labelling $f_5: E \mapsto L$ so that $f_5$ is $g$-antimagic, and for all $v\in V$ we have $s_{(f_5,g)}(v)\not\equiv 0 \pmod{k_1k_2}$.
\end{lemma}
\begin{proof}
 First we apply Lemma \ref{lem_E4_label} to $G$ --- let $f_4:E_1\cup E_2\cup E_3 \cup E_4 \mapsto L$ be a labelling satisfying the conclusions of that lemma, and let $L'$ be the labels not in the image of $f_4$; so $L'\subseteq L_4$. We have $|L'|=|E_5|$. For a vertex $v\in V$, let $g'(v)=s_{(f_4,g)}(v)$. Now, we apply Lemma \ref{lem_star_label} to the stars $S_1,\dots,S_k$ which make up $F_S$, and the label set $L'$, with the function $g':\bigcup_{i=1}^k V(S_i)\mapsto \mathbb N$ for the function $g$. This gives us a bijective labelling $f'_5:E_5\mapsto L'$ so that $s_{(f'_5,g')}(c_i)\neq s_{(f'_5,g')}(c_j)$ for $1\le i < j \le k$. Now, let $f_5$ be the labelling given by $f_4$ on $E_1\cup E_2\cup E_3 \cup E_4$ and $f'_5$ on $E_5$. We claim that $f_5$ is a $g$-antimagic labelling, with no sum $s_{(f_5,g)}(v)$ divisible by $k_1k_2$. Since $f'_5$ labels no edge incident with $V_0$, and all the labels used by $f'_5$ are divisible by $k_1k_2$, the following conditions hold from the properties of $f_4$:
\begin{enumerate}
\item if $v$ is a vertex of $V_2$ which is the centre $c_i$ of one of the stars $S_i$, then $s_{(f_5,g)}\equiv 1 \pmod{k_1k_2}$,
\item if $v$ is a vertex of $V_2$ which is not the centre of one of the stars $S_i$, then $s_{(f_5,g)}\equiv m \pmod{k_1k_2}$,
\item if $v$ is a vertex of $V_1$, then $s_{(f_5,g)}\equiv 1 \pmod{k_1k_2}$,
\item if $v$ is a vertex of $V_0$, $s_{(f_5,g)}$ is not congruent to $0$, $1$ or $m$ modulo $k_1k_2$,
\item for distinct vertices $v_1$ and $v_2\in V_0$, $s_{(f_5,g)}(v_1)\neq s_{(f_5,g)}(v_2)$.
\end{enumerate}
Since by Condition \ref{Con_Struct} of Lemma \ref{lem_partition} all the vertices of $V_1$ are the centres of stars, all centres $c$ of stars in $F_S$ have $s_{(f_5,g)}(c)\equiv 1 \pmod{k_1k_2}$, and all the other vertices $v$ in $V_2$ have $s_{(f,g)}(v)\equiv m \pmod{k_1k_2}$, whereas all vertices $v$ in $V_0$ have $s_{(f,g)}(v)\notin \{0,1,m\} \pmod {k_1k_2}$. So it suffices to check that none of these three sets contain two vertices with equal sums. For two vertices in $V_0$, this is the last condition above. For two vertices $c_i$ and $c_j$ which are the centres of stars $S_i$ and $S_j$, from our application of Lemma \ref{lem_star_label} we have $s_{(f_5,g)}(c_i)= s_{(f'_5,g')}(c_i)\neq s_{(f'_5,g')}(c_j)= s_{(f_5,g)}(c_j)$. For two vertices $v_1$ and $v_2$ in $V_2$ which are not the centres of stars, each has exactly one edge in a star, and so we have $s_{(f_5,g)}(v_1)=s_{(f_4,g)}(v_1)+l_1$, and $s_{(f_5,g)}(v_2)=s_{(f_4,g)}(v_2)+l_2$, for some $l_1\neq l_2$ in $L_4$, and hence in $[1,k_1k_2n]$. However, from Lemma \ref{lem_E4_label} we also have $s_{(f_4,g)}(v_1)\equiv s_{(f_4,g)}(v_2)\equiv m \pmod{k_1k_2n}$. Hence $s_{(f_5,g)}(v_1)\neq s_{(f_5,g)}(v_2)$.
\end{proof}
From the bound on $z(k,l)$ given by Theorem \ref{thm_total_dom}, it is easy to see that \eqref{eq_for_delta} holds for sufficiently large $\delta$ so long as $\frac{2}{\min(k_1-4,k_2-3)}<1/2$. This establishes Lemma \ref{lem_min_suffices} for $k_1$, $k_2\ge 9$, in which case we can take $c$ to be $2k_1k_2+k_2$ and $\delta$ to be the minimal integer satisfying \eqref{eq_for_delta}. In fact, it can be calculated that the best bound on $d$ is achieved when $(k_1, k_2) = (13,11)$, for which Lemma \ref{lem_min_suffices} holds with constants $c=2k_1k_2+k_2=297$ and $\delta = 1663$. Then from our proof of Theorem \ref{thm_main} from Lemma \ref{lem_min_suffices}, Theorem \ref{thm_main} holds with $d_0=4182$.

\section{Further Work}\label{sec_further}
In this section, we discuss possible directions for work on antimagic graphs. The main open question remains the conjecture of Hartsfield and Ringel, that all connected graphs other than $K_2$ are antimagic. However, Theorem \ref{thm_main} does not require $G$ to be connected, leading us to pose the following question:
\begin{question}
What is the least real number $d_0$ such that any graph with average degree at least $d_0$ with no isolated edges and at most one isolated vertex is antimagic?
\end{question}
Our proof gives an upper bound on $d_0$ of $4182$. While there may be room for decreasing this bound by proceeding more carefully, it seems unlikely that an approach similar to the one employed here will bring the bound below, say, $1000$. For a lower bound, it is easy to see that if $G$ has no isolated vertices and 
\begin{equation}\label{eq_not_anti}
|E|(|E|+1)<|V|(|V|+1)/2,
\end{equation}
then $G$ is not antimagic --- indeed, the total of the induced sums at all the vertices is not large enough for the vertex sums to be distinct positive integers. This gives $d_0\ge \sqrt{2}$. We conjecture that in fact any graph with no isolated edges and at most one isolated vertex not satisfying \eqref{eq_not_anti} is antimagic, and hence that $d_0 = \sqrt{2}$.

Another direction arises from the observation that our proof of Theorem \ref{thm_main} allows us to construct antimagic labellings of graphs $G$ with large average degree with many more label sets than just $[1,|E(G)|]$. In fact, we approximately need one label in each congruency class modulo $k_1k_2n$, and a further $n$ in each congruency class modulo $k_1k_2$. This leads us to ask whether we could use any label set. Explicitly, we call a graph $G=(V,E)$ \emph{label-antimagic} if for any set $L$ of positive integers with $|L|=|E|$ there is a bijective antimagic labelling $f:E\mapsto L$.
\begin{question}
Is there some constant $d_l$ such that all graphs with average degree at least $d_l$ with no isolated edges and at most one isolated vertex are label-antimagic?
\end{question}

\section{Acknowledgements}
I would like to thank Karen Gunderson, who introduced me to this problem and with whom I had several helpful conversations about antimagic colourings. I would also like to thank B\'ela Bollob\'as and both the anonymous referees, whose comments have made the proof clearer and stronger.

\end{document}